\documentclass[preprint,12pt]{elsarticle}
\usepackage[english]{babel}
\usepackage[utf8x]{inputenc}
\usepackage{listings}
\usepackage{color}
\usepackage{algorithm}
\usepackage{algorithmic}
\usepackage{verbatim}
\usepackage{soul} 
\usepackage[toc,page]{appendix}
\usepackage{amsfonts}
\usepackage{caption}
\usepackage{subcaption}

\usepackage{chngcntr}
\usepackage{apptools}
\usepackage{ntheorem}
\usepackage{multirow}

\newtheorem{theorem}{Theorem}[section]
\newtheorem{definition}{Definition}[section]
\newtheorem{example}{Example}[section]

\usepackage{float}
\usepackage{shuffle}

\AtAppendix{\renewtheorem{theorem}{Theorem A.}}

\newcommand{\tri}[1]{{\left\vert\kern-0.25ex\left\vert\kern-0.25ex\left\vert #1 
    \right\vert\kern-0.25ex\right\vert\kern-0.25ex\right\vert}}

\usepackage{eqparbox}
\renewcommand{\algorithmiccomment}[1]{\hfill\eqparbox{COMMENT}{$\triangleright$ #1}}

\usepackage{amsmath}
\usepackage[colorinlistoftodos]{todonotes}
\usepackage[colorlinks=true,
            linkcolor = red,
            urlcolor=blue]{hyperref}
\usepackage{listings}
\usepackage{url}
\usepackage{graphicx}
\usepackage{subcaption}
\usepackage{cleveref}

\usepackage{amssymb}

\usepackage{booktabs}

\newsavebox{\measurebox}
\newenvironment{proof}{\paragraph{Proof:}}{\hfill$\square$}

\begin{document}

\begin{frontmatter}

\title{Towards fast weak adversarial training to solve high dimensional parabolic partial differential equations using XNODE-WAN}

 \author[ucl]{Paul Valsecchi Oliva}
 \author[oxford,ATI,clyde]{Yue Wu}
    \author[texas]{Cuiyu He}
      \author[ucl,ATI,*]{Hao Ni}
 \affiliation[ucl]{organization={Mathematics Department, University College London},
             addressline={25 Gordon St},
             city={London},
             postcode={WC1H 0AY},
             country={UK}}
            
\affiliation[oxford]{organization={Mathematical Institute, University of Oxford},
            addressline={ Radcliffe Observatory, Andrew Wiles Building, Woodstock Rd }, 
            city={Oxford},
            postcode={OX2 6GG}, 
            country={UK}}

 \affiliation[ATI]{
             organization={Alan Turing Institute},
             addressline={2QR, John Dodson House, 96 Euston Rd},
             city={London},
             postcode={NW1 2DB},
             country={UK}}            
\affiliation[clyde]{organization={The Department of Mathematics and Statistics, University of Strathclyde},
            addressline={26 Richmond St}, 
            city={Glasgow},
            postcode={G1 1XQ}, 
            country={UK}}
\affiliation[texas]{organization={School of Mathematics and Statistical Science, University of Texas Rio Grand Valley},
            addressline={1 West University Blvd}, 
            city={Brownsville,TX},
            postcode={78520}, 
            country={USA}}
\affiliation[*]{Corresponding author:h.ni@ucl.ac.uk}
\begin{abstract}
Due to the curse of dimensionality, 
solving high dimensional parabolic partial differential equations (PDEs) has been a challenging problem for decades. Recently, a weak adversarial network (WAN) proposed in (Y.Zang et al., 2020) offered a flexible and computationally efficient approach to tackle this problem defined on arbitrary domains by leveraging the weak solution. WAN reformulates the PDE problem as a generative adversarial network, where the weak solution (primal network) and the test function (adversarial network) are parameterized by the multi-layer deep neural networks (DNNs). However, it is not yet clear whether DNNs are the most effective model for the parabolic PDE solutions as they do not take into account the fundamentally different roles played by time and spatial variables in the solution. To reinforce the difference, we design a novel so-called XNODE model for the primal network, which is built on the neural ODE (NODE) model with additional spatial dependency to incorporate the a priori information of the PDEs and serve as a universal and effective approximation to the solution.  The proposed hybrid method (XNODE-WAN), by integrating the XNODE model within the WAN framework, leads to significant improvement in the performance and efficiency of training. Numerical results show that our method can reduce the training time to a fraction of that of the WAN model. 

\end{abstract}

\begin{keyword}
parabolic partial differential equation \sep generative adversarial network \sep neural ordinary differential equation \sep weak adversarial network
\end{keyword}

\end{frontmatter}


\section{Introduction}

Applying classical numerical methods (finite element, finite difference, finite volume, etc) to solve high dimensional partial differential equations (PDEs) has been highly challenging due to the notorious curse of dimensionality \cite{han2018solving, hutzenthaler2020overcoming}. Traditional numerical methods, used as PDE solvers, lead to an exponential growth of computing cost with respect to the dimension of the PDE problem. The application of neural networks on high dimensional data, including image classification \cite{druzhkov2016survey} and natural language processing \cite{goldberg2016primer}, has empirically proven to be successful. Recently, applying deep learning models to tackle high dimensional PDE problems, which achieved superior empirical results
\cite{han2018solving,raissi2019physics,li2020multipole,he2020mesh,lu2021deepxde}, has been a growing area of research. Relevant numerical analysis that takes into account of approximation rates, generality and optimization theory can be also found in many works, see \cite{berner2020analysis,duan2021convergence,hutzenthaler2020overcoming,luo2020two}. 

As an important class of PDEs, parabolic PDEs have been studied extensively and applied to a various number of fields, including acoustic propagation \cite{gilbert1989application}, heat conduction \cite{douglas1956numerical}, image denoising \cite{hahn2011orientation}, derivative pricing \cite{achdou2005computational}, etc. In \cite{dgmsirignano2018, zang2020weak}, multi-layer deep neural networks (DNNs) are used to approximate the solution to parabolic problems thanks to their universality. However, DNNs do not take into account the fundamental difference in the roles played by time and spatial variables in the parabolic solution, and hence may not be efficient in capturing the spatio-temporal dependence of the parabolic solution. To this end, we propose a novel so-called XNODE model for the solution to reinforce the spatio-temporal dependence. It effectively incorporates the a priori information of the PDEs through the modified neural ordinary differential equations (NODEs)\cite{chen2018neural}.

More specifically, we consider the following parabolic PDE defined on an arbitrary bounded domain $\mathcal{D} \subset [0, T] \times \mathbb{R}^{d}$ under the mild condition, 
\begin{align}\label{pde}
\begin{cases}
\partial_t u(t,\mathbf{x})-\overset{d}{\underset{i=1}{\sum}}\partial_i(\overset{d}{\underset{i=1}{\sum}}a_{ij}(t,\mathbf{x})
\partial_ju(t,\mathbf{x}))+\overset{d}{\underset{i=1}{\sum}}b_i(t,\mathbf{x})\partial_iu(t,\mathbf{x}) &\\
+c(u, t,\mathbf{x})-f(t,\mathbf{x})=0& \text{for } (t, \mathbf{x}) \in \mathcal{D},\\
u(t, \mathbf{x})= g(t,\mathbf{x}) & \text{on }\partial \mathcal{D},\\
u(0,\mathbf{x})-h(\mathbf{x})=0 & \text{on }\Omega(0).
\end{cases}
\end{align}
where $a_{ij},b_i, f:\mathcal{D}\to\mathbb{R}$, $c: \mathbb{R} \times \mathcal{D} \to \mathbb{R}$ and $h: \Omega(0) \to \mathbb{R}$ are given. $c$ can be a non-linear function with respect to the first augment. Let $\Omega(t):= \{\mathbf{x} | (t, \mathbf{x}) \in \mathcal{D}\}$ denote the spatial domain of $\mathcal{D}$ when restricting time to be $t$. Note that $\mathcal{D}$ could be not only the time-independent domain, i.e. there exists $\Omega \subset \mathbb{R}^{d}$, such that $\mathcal{D} = [0, T] \times \Omega$ , but also the time varying domain, i.e.  $\Omega(t)$ changes over time $t$. 

The motivation of the XNODE model lies in the key observation that for a fixed spatial variable $\mathbf{x}$, a parabolic PDE solution $\tilde{u} = u(., \mathbf{x})$ evolves into an ODE solution, which naturally leads to employing  the NODE network \cite{chen2018neural}. A NODE network builds a general model by leveraging the ODE solver and parameterising the vector field as a neural network, which can be regarded as a continuous version of the residual neural network. The NODE model recently introduced in \cite{chen2018neural} soon became a top-pick model for (continuous) time series modelling because of its efficiency and scalability. As opposed to classical neural networks models, the gradient calculation of the NODE models via the adjoint method has constant memory cost.

Therefore, to model the parabolic PDE solution efficiently, the above observation leads to the proposed XNODE model, which is built by introducing the additional dependency on the spatial variable $\mathbf{x}$ to the NODE model. More specifically, the XNODE model takes any spatial variable $\mathbf{x}$ and $h$ from \eqref{pde} as the input, and returns the output of the NODE with vector fields dependent on spatial variable $\mathbf{x}$ to predict the solution $\tilde{u} = u(., \mathbf{x})$. In contrast to the DNNs which treat the time and spatial variable equally, the proposed XNODE model utilizes the NODE to capture the temporal dependency of the solution path $\tilde{u}$, while the spatial dependence is embedded to the vector field of NODE model. In the XNODE model, the initial condition of PDEs $h$ is incorporated effectively through the initial condition of the NODE. Moreover, we introduce efficient data sampling and path construction with the XNODE method to ensure efficiency for predicting the solution on both time-independent domains and time-varying domains.  


The formulation of the PDE  problem using deep learning can be broadly divided into two classes, (1) optimisation problems and (2) min-max problems. For the first category, learning the PDE solution is translated into minimizing the loss function, which is used to quantify the performance of the estimated solution. For example, one can consider the mean square error regarding the equilibrate equations and boundary and initial conditions as the loss function, see e.g., \cite{cai2020deep,dgmsirignano2018, he2020mesh}. The least square approach is known for its generality. However, the evaluation of high order derivatives are required in the cost functional if the equilibrate equation is directly penalized  \cite{dgmsirignano2018, he2020mesh}. Another commonly used cost functional is based on the energy or minimal principal, \cite{weinan2018deep,samaniego2020energy,hong2021priori}. Though the energy functional often only involves the lower order derivatives, i.e. no higher derivative is required to compute, the energy based method is restricted to certain differential operators for which a minimal principal holds. The second category is to utilize the generative adversarial networks (GANs) to solve the PDEs by using the min-max game formulation. The weak adversarial network (WAN) \cite{zang2020weak} is a typical example of this type. By leveraging the weak formulation of the PDE solution, WAN converts the task of finding the weak solution into a saddle point problem with the solution $(u, \phi)$, where $u$ is the solution to the PDE, and $\phi$ is a test function. Compared with the least square method of the first category, WAN can be applied to solve a large general class of PDEs, including the cases where the strong solution does not exist as long as the weak solution is well defined. It also can be applied to the case where the minimal principal of the PDEs is not satisfied. ax optimization, it may bring some computational benefits resulted from only requiring lower order derivative computation. 

Based on the above considerations, we propose the hybrid XNODE-WAN method by incorporating the XNODE model into the WAN framework \cite{zang2020weak} as an enhancement of the WAN to learn the solutions to parabolic PDEs more efficiently. In the WAN framework, there is a primal network and an adversarial network to approximate the solution ($u$) and the test function ($\phi$) respectively, which are both approximated by DNNs\cite{zang2020weak}. In our proposed XNODE-WAN method, we replace the DNNs with the XNODE model for the primal network of WAN, as the XNODE model is able to better model the spatio-temporal dependency of the solution and encode a priori information of the PDE, which leads to reduced training time and faster convergence. As many of the other deep learning methods for PDE solvers \cite{zang2020weak, he2020mesh, hong2021priori}, the proposed XNODE-WAN method is mesh-free. It can be applied to a general class of the spatio-temporal domains, whose spatial domain can be irregular and/or time-varying.

Numerical results show that our proposed method significantly outperforms WAN in terms of training efficiency for both time-independent and time-varying domains. Moreover, such improvement of the XNODE-WAN model is consistent regardless of the dimension of the spatial variable, which demonstrates a better scalability for the high dimensional PDEs.    

The paper is structured as follows. Section \ref{sec:WAN} gives a brief introduction to the WAN framework \cite{zang2020weak}. It is followed by Section \ref{sec:NODE-WAN}, which starts with preliminaries of NODE model and then introduces the proposed XNODE model for the primal solution and the hybrid XNODE-WAN method for parabolic PDE solution on general time-space domains. Numerical results are provided in Section \ref{sec:numerical}. Section \ref{sec: Conclusion} concludes the paper.

\section{Weak Adversarial Network (WAN) framework}\label{sec:WAN}
In this section, we provide a brief summary of theoretical underpinning and a general framework of the weak adversarial network for solving the linear and non-linear parabolic PDEs of form \eqref{pde}. 
\subsection{The weak formulation}\label{sec:weakformula}

We consider the PDE problem \eqref{pde} on a time-dependent bounded domain denoted by
$\mathcal{D} \subset [0, T] \times \mathbb{R}^{d}$ with $\Omega(t)= \{\mathbf{x} | (t, \mathbf{x}) \in \mathcal{D}\}$,
where we impose typical assumptions such as the uniformly parabolic property of the differential operator \footnote{The uniformly parabolic property means
there exists $\alpha>0$ such that 
$	\sum_{i,j =1}^d a_{ij}(x,t) \zeta_i \zeta_j \ge \alpha | \zeta|^2 \quad \forall \zeta \in \mathbb{R}^n, (x,t) \in [0,T] \times \Omega.$}. 
We assume further that $a_{ij} \in L^{\infty}(\mathcal{D})$, $f \in L^2(\mathcal{D})$ and $h \in L^2(\Omega(0))$.

When the problem \eqref{pde} is linear and on a time-independent domain, i.e., $\mathcal{D} = [0,T] \times \Omega$,  the well-posedness of equation \eqref{pde} can be ensured with the solution $u \in L^2([0,T], H_0^1(\Omega)):=\{ v: v \in L^2([0,T], H^1(\Omega), v|_{\partial \Omega} = 0  \}$ and $\partial_t u \in L^2([0,T],H^{-1}(\Omega))$, given the coefficients are uniformly elliptic \cite{renardy2006introduction}. 

In general, we consider the time-space domain $\mathcal{D} \subset [0,T]  \times \mathbb{R}^d$. We now define the following bilinear form: 
\begin{align*}
    	B(u,v) &= \int_0^T \int_{\Omega(t)} \partial_t u \cdot v \mathrm{d}x\mathrm{d}t\\
    	&+  \int_0^T \int_{\Omega(t)} 
	\left(\sum_{i j} a_{ij} \partial_j u \partial_i v + \sum_i b_i \partial_i u, v + cu v \right)\mathrm{d}x \mathrm{d}t
\end{align*}
and the linear operator:
\[
	f(v) = \int_0^T \int_{\Omega(t)} f v \mathrm{d}x \mathrm{d}t.
\]

 For the ease of notation, we define $H_0^{1}(\mathcal{D})$, $L^{2}(\mathcal{D})$ and $L^{2}(\partial \mathcal{D})$ by
\begin{eqnarray*}
H_0^{1}(\mathcal{D}) := \Bigg\{ f: \mathcal{D}  \rightarrow \mathbb{R} \Bigg\vert&& \forall t \in [0, T], f(t, \cdot) \in H_0^{1}\big(\Omega(t)\big)  \text{ and }\\
&&\int_{0}^{T} \int_{\Omega(t)}|f(t,x)|^2 + | \nabla f(t,x)|^2dxdt < + \infty \Bigg\},
\end{eqnarray*}
\begin{eqnarray*}
L^{2}(\mathcal{D}) := \Bigg\{ f: \partial\mathcal{D}  \rightarrow \mathbb{R} \Bigg\vert \int_{0}^{T} \int_{\Omega(t)}|f(t,x)|^2 dxdt < + \infty \Bigg\},
\end{eqnarray*}
and
\begin{eqnarray*}
L^{2}(\partial\mathcal{D}) := \Bigg\{ f: \mathcal{D}  \rightarrow \mathbb{R} \Bigg\vert \int_{0}^{T} \int_{\partial\Omega(t)}|f(t,x)|^2 dxdt < + \infty \Bigg\},
\end{eqnarray*}
where $\nabla f(t,\cdot)$ is the weak derivative of $f(t,\cdot)$. 

From integration by parts, it is easy to check that the weak solution $u$ to \eqref{pde} satisfies the following variational problem: 
\begin{equation}\label{eqn:equal}
	B(u,v)  = f(v) \quad \forall v \in  H^1_0(\mathcal{D}).
\end{equation}
Therefore, it leads to a new interpretation of $u$ as a solution to the below minimization problem: 
\begin{equation}\label{inf-sup}
	\inf_{w \in H^1_0(\mathcal{D})} \sup_{v \in H^1_0(\mathcal{D}), v \neq 0} \dfrac{|B(w,v) -f(v)|^2}{\|v\|^2_{L^2(\mathcal{D})}},
\end{equation}
where $\|v\|_{L^2(\mathcal{D})}$ is $L^2$ norm of $v$(\cite{zang2020weak}). Note that for a given $w$, the optimal test function $v$ is the one to attain the supremum $|B(w,v) -f(v)|^2/\|v\|^2_{L^2(\mathcal{D})}$. In \eqref{inf-sup} we reformulate the original parabolic PDE into an inf-sup type optimisation problem, which sets the foundation for the WAN framework.





\subsection{WAN framework }\label{sec:wan}

In this subsection, we introduce the WAN framework initially proposed in \cite{zang2020weak} (cf Algorithm 3 in \cite{zang2020weak}) for solving parabolic PDEs by leveraging the weak formulation in Section \ref{sec:weakformula}. The main idea is to formulate the problem as the inf-sup problem \eqref{inf-sup}, where the PDE solution $u$ and the test function $\phi$ are parameterised by deep neural networks $u_\theta$ and $\phi_\eta$, respectively, with  $\theta, \eta$ being the trainable model weights. In \cite{zang2020weak}, a deep neural network is employed for both the solutions $u$ and $\phi$. The primal network will be replaced by our proposed XNODE model, which will be discussed in in Section \ref{sec:NODE-WAN}.


The loss function of \cite{zang2020weak}, which is also used in our method, consists of the interior loss, the boundary loss, and the initial loss as follows:
\begin{equation}\label{loss}
L(\theta, \eta) = L_\text{int}(\theta, \eta) + \alpha L_\text{bdry}(\theta) + \gamma L_\text{init}(\theta),
\end{equation}
where $\alpha$, $\gamma$ are hyperparameters as penalty terms and
\begin{align}\label{L_int}
\begin{split}
L_{\text{int}}(\theta, \eta)& = \log\left(\frac{|B(u_\theta, \phi_\eta) - f(\phi_\eta)|^2}
{||\phi_\eta||_{L^2(\mathcal{D})}^2}\right),\\
L_\text{init}(\theta)&= ||u_\theta(0,\mathbf{x})-h(\mathbf{x})||_{L^2(\Omega(0))}^2,\\
L_\text{bdry}(\theta)&= ||u_\theta(t,\mathbf{x})-g(t,\mathbf{x})||_{L^2(\partial \mathcal{D})}^2.
\end{split}
\end{align}
Note that the interior loss $L_{\text{int}}$ is based on the weak formulation of the solution \eqref{inf-sup} to quantify the discrepancy between the true solution $u$ and the estimator $u_{\theta}$ over the interior. Due to the different norms and log operator applied in the loss function, the coefficient $\alpha$ and $\gamma$ should be chosen numerically by cross-validation to ensure efficiency. 

The evaluation of three loss terms can be done through the Monte-Carlo simulation on the domain. We shall create a collection of points for such evaluation at every iteration of our algorithm. As shown in \cite{zang2020weak}, for any bounded domain $\mathcal{D}$, at each iteration, we uniformly sample $N_n$ points $(t_i, x_i)_{i = 1}^{N_n}$ over the interior of the domain. Similarly, we conduct the uniform sampling on the boundary $\partial \mathcal{D}$ with the corresponding number of sampling points denoted by $N_b$. Let $ \tilde L_\text{int}(\theta, \eta), \tilde L_\text{bdry}(\theta)$ and $\tilde L_\text{init}(\theta)$ denote the corresponding Monte-Carlo estimators of $L_\text{int}(\theta, \eta), L_\text{bdry}(\theta)$ and $L_\text{init}(\theta)$ based on the sampling points respectively. Then we define the empirical loss function by 
\begin{equation}\label{loss-approx}
\tilde L(\theta, \eta) = \tilde L_\text{int}(\theta, \eta) + \alpha \tilde L_\text{bdry}(\theta) + \gamma \tilde L_\text{init}(\theta).
\end{equation}
Eq. \eqref{loss-approx} is the loss function of the min-max game used in the WAN to learn the PDE solution.

Despite the same loss function of our method, due to the XNODE model, we adjust the above point sampling strategies to both time-independent domain and time-dependent domain respectively. One may refer to Section \ref{sec:independent} and \ref{subsec:NODE-WAN-b}.


The algorithm of WAN is outline in Algorithm \ref{WAN_alg}, where the lines with additional comments will be modified in our algorithm in Section \ref{sec:NODE-WAN}. A list of notation explaination can be found in Table \ref{table2}.

\begin{algorithm}[H]
\caption{WAN Algorithm}\label{WAN_alg}
\textbf{Input:} domain $\mathcal{D}$, tolerance $\epsilon>0$, $N_\mathcal{D}/N_{\partial \mathcal{D}}/N_0$: number of sampled points on the domain/boundary/initial condition,  $K_u/K_\phi$: number of solution/adversarial network parameter updates per iteration, $\alpha/\gamma \in \mathbb{R}^+:$ the weight of boundary/initial loss,  and the hyperparameters for all DNNs involved
\begin{algorithmic}[1]
\STATE{\textbf{Initialise:} the DNN network architectures $u_\theta,\phi_\eta:\mathcal{D} \to \mathbb{R}$} 
\STATE{generate points sets $X_\mathcal{D}$, $X_{\partial \mathcal{D}}$ and $X_0$} \\
\algorithmiccomment{random point sampling step}

\WHILE{$\tilde L(\theta,\eta)>\epsilon$} 
\STATE{\# \texttt{update weak solution network parameter}}
\FOR{$k = 1,...,K_u$}
\STATE{compute $\nabla_\theta \tilde L(\theta, \eta)$;}\qquad  \algorithmiccomment{gradient calculation step }
\STATE{update $\theta\gets\theta- \tau_\theta \nabla_\theta \tilde L(\theta, \eta)$;} 
\ENDFOR
\STATE{\# \texttt{update test network parameter}}
\FOR{$k = 1,...,K_\phi$}
\STATE{compute $\nabla_\eta \tilde L(\theta, \eta)$;}\qquad \algorithmiccomment{gradient calculation step}
\STATE{update \small{$\eta\gets\eta+\tau_\eta \nabla_\eta \tilde L(\theta, \eta)$;}
 }
\ENDFOR
\STATE{generate points sets $X_\mathcal{D},X_{\partial \mathcal{D}}$} and $X_0$;\\  \algorithmiccomment{random point sampling step for the next run}

\ENDWHILE\\
\textbf{Output:} the weak solution $u_\theta:\mathcal{D} \to\mathbb{R}$
\end{algorithmic}

\end{algorithm}
\section{The XNODE-WAN method}\label{sec:NODE-WAN}
  \newcommand{\bF}[1]{\mathbf{#1}}
In this section, we introduce a novel 
so-called XNODE model for the solution $u$ to the parabolic PDE problem \eqref{pde} on arbitrary spatio-temporal domains. It can be conveniently incorporated within the WAN framework by replacing the deep neural network by the XNODE model for the primal solution to achieve superior training efficiency. We will refer to our method as XNODE-WAN in the remaining contents. 

The main motivation of the proposed XNODE model comes from the observation that with the spatial variable fixed, the parabolic problem \eqref{pde} reduces to an ODE problem, which can be efficiently modelled using the neural ODE model (NODE) \cite{chen2018neural}. Hence, the proposed XNODE model captures the additional dependency on spatial coordinates, incorporates the a priori information of the parabolic PDEs and initial conditions effectively. To facilitate the efficiency of XNODE to approximate the primal solution over the domain, we propose the use of multiple constant paths to traverse the domain as the point sampling strategy over the domain.

For the rest of this section,
we start with a brief introduction to conventional NODE model. Then we proceed to discuss how to adapt the NODE model to approximate the parabolic solution by introducing additional spatial state variable, i.e., the XNODE model. This is followed by a subsection of using the XNODE-WAN model for solving time-independent domain problems, which involves incorporating the XNODE model within the WAN framework. We then generalize the proposed methodology to the time varying domains (see Figure \ref{fig:paths} for illustrative examples of time-independent domains and time-varying domains respectively).

\begin{figure}[H]
\centering
\begin{subfigure}{.485\textwidth}
\includegraphics[width=\textwidth]{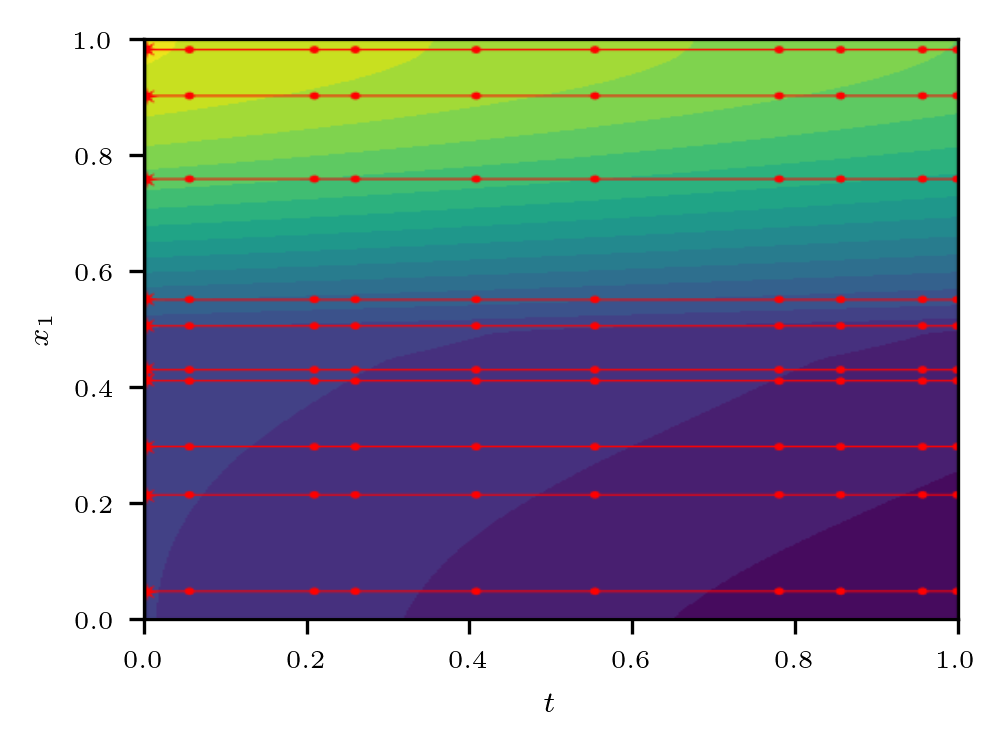}
\caption{a time-independent domain}\label{fig:path-a}

\end{subfigure}
\begin{subfigure}{.485\textwidth}
\includegraphics[width=\textwidth]{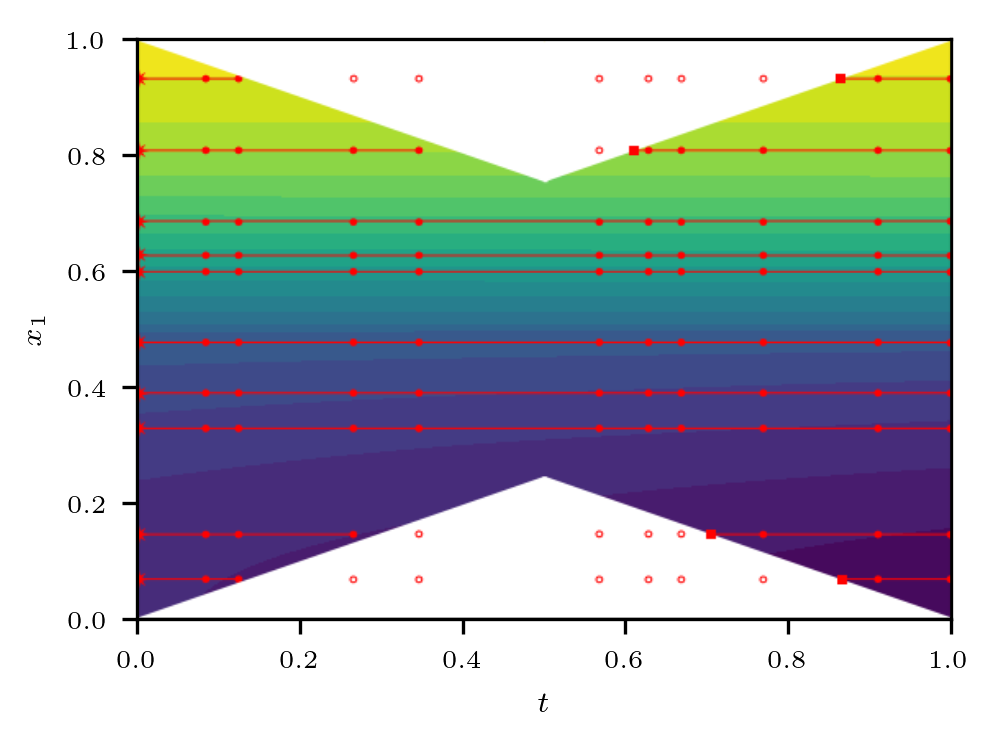}
\caption{a time-dependent domain.}\label{fig:path-b}
\end{subfigure}
\caption{An illustration of \emph{constant paths} for spatial-time domains. (a) The red stars are the initial time points and the points are all the points at which we evaluate the data. Note that they are aligned in the time dimension. (b) The red squares here represent the inserted points that act as the initial points for the paths that cannot be traced back to the initial time. }

\label{fig:paths}
\end{figure}
  
\subsection{The neural ODE (NODE) model}\label{sec:node}

The NODE \cite{chen2018neural} method that fuses concepts of ODEs and neural networks, is well known for offering benefits to time series and density modeling. One of its main benefits lies in the computational feasibility brought by the adjoint sensitivity method, for which one solves an ODE to compute the gradients with constant memory cost. 

More specifically, NODE models a continuous function which maps $\mathbb{R}^{d}$ to $\mathbb{R}^{d}$ (or $\mathcal{C}([0, T], \mathbb{R}^{d})$\footnote{$\mathcal{C}([0, T], \mathbb{R}^{d})$ denotes the space of continuous functions mapping from $[0, T]$ to $\mathbb{R}^{d}$.}) by defining the ODE model as follows: for every input $\mathbf{h}_0 \in \mathbb{R}^{d}$, the corresponding output of NODE model is $\textbf{h}(T)$ (or $\textbf{h} = (\textbf{h}(t))_{t \in [0, T]}$) where $\textbf{h}$ satisfies 
 \begin{eqnarray}\label{ODE}
\frac{\mathrm{d}\textbf{h}(t)}{\mathrm{d}t}=\mathcal{N}^{\text{vec}}_{\theta}(\textbf{h}(t), t), \quad \mathbf{h}(0) = \mathbf{h}_0.
\end{eqnarray}
Here, the vector field $\mathcal{N}^{\text{vec}}_\theta$ is a paramerized neural network with trainable weights $\theta$. The generality of the NODE model mainly come from the fact that the vector field can be universally approximated by a neural network. 
By Picard's Theorem, the initial value problem will yield a unique solution provided $\mathcal{N}_{\theta}^{\text{vec}}$ is Lipschitz continuous which is guaranteed by using finite weights and the activation functions $\texttt{relu}$ or $\texttt{tanh}$ \cite{chen2018neural}. 

The details of forward evaluation of $\mathbf{h}$ and the derivative calculation of $\mathbf{h}$ with respect to $\theta$  can be found in \cite{chen2018neural}. Note that the dependency of the NODE output $\mathbf{h}$ on the input is solely passed through the initial condition of ODE model, i.e.,  $\mathbf{h}_0$.

\subsection{XNODE-WAN}\label{sec:node_fixed_spatial_domain}
\subsubsection{Time-independent domain}\label{sec:independent}

Consider the time-independent domain $[0,T] \times \Omega$, where $\Omega \subset \mathbb{R}^d$ is bounded. When given a fixed spatial variable $\mathbf{x} \in \Omega$, the map $\tilde{u}: t \mapsto u(t, \mathbf{x})$ can be viewed as a solution to the following ODE problem from PDE \eqref{pde}:
 \begin{equation}\label{x-ODE_constant}
 	 \frac{\mathrm{d}\tilde{u}(t)}{\mathrm{d}t} = F(\tilde{u},t; \mathbf{x}), \quad \tilde{u}(0) = h(\mathbf{x}),
 \end{equation} 
 where $F$ is a function depending on $(\tilde u,t; \mathbf{x})$ defined as  
 \begin{eqnarray}\label{PDE_ODE}
 F(\tilde u, t; \mathbf{x}) = \sum_{i=1}^d\partial_i\Big(\sum_{j=1}^da_{ij}\partial_ju(t,\mathbf{x})\Big)+\sum_{i=1}^db_i\partial_iu(t,\mathbf{x})+c(u(t,\mathbf{x}))-f(t,\mathbf{x}).
 \end{eqnarray} 
Note that although $F$ is not observable due to the unknown partial derivative  $\partial_i u(t,\mathbf{x})$  , it can be universally approximated by a neural network provided that $F$ is continuous. This observation motivates us to use the NODE model \cite{chen2018neural} for a primal solution $\tilde{u}$ for each given $\mathbf{x}$. 

\textbf{XNODE Model}: Similar to the NODE model, we will employ a deep neural network $\mathcal{N}^{\text{vec}}$ to approximate the vector field $F$ in \eqref{x-ODE_constant}.
To construct a universal model for the primal solution $\tilde{u}$, we suggest in \eqref{x-ODE_constant} that the spatial variable $\mathbf{x}$ should be incorporated in the vector field $\mathcal{N}^{\text{vec}}$ as the additional dependent variable. More specifically, the XNODE model maps an arbitrary input $\mathbf{x} \in \mathbb{R}^d$ to the output $o_\mathbf{x}(t)_{t \in [0, T]} \in \mathcal{C}([0, T], \mathbb{R})$ through the $\mathbb{R}^{h}$-valued latent process $\mathbf{h} = (\mathbf{h}(t))_{t \in [0, T]}$
which is defined by the output of the NODE model with the $\mathbf{x}$-dependent vector field. More precisely, for every $ t \in [0, T]$, we have
 \begin{eqnarray*}
 \begin{cases}
 \frac{\mathrm{d}\mathbf{h}(t)}{\mathrm{d}t} = \mathcal{N}^{\text{vec}}_{\theta_2}(\mathbf{h}(t), t, \mathbf{x}), \quad \mathbf{h}(0) = \mathcal{N}^{\text{init}}_{\theta_{1}}(h(\mathbf{x})) \in \mathbb{R}^h.& \\
o_\mathbf{x}(t)= \mathcal{L}_{\tilde{\theta}}(\mathbf{h}(t)). &
 \end{cases}
 \end{eqnarray*}
where $\mathcal{N}^{\text{vec}}_{\theta_2}$ and $\mathcal{N}^{\text{init}}_{\theta_{1}}$ are neural networks fully parameterized by $\theta_{1}$ and $\theta_2$ for the vector fields and initial condition of the hidden neural $\mathbf{h}$ respectively, and $\mathcal{L}_{\tilde{\theta}}$ is a linear trainable layer with the weights $\tilde{\theta}$. We denote by $\Theta = (\theta_1, \theta_2, \tilde{\theta})$ the set of all trainable model parameters of the proposed XNODE model. We note that the output of XNODE  $o$ is a function, that lives in an infinite dimensional space. However in practice, the model output has to be finite dimensional for numerical computation. Hence given any time partition $\Pi_{T} = (t_i)_{i = 1}^{n_T}$\footnote{$\Pi_{T}$ is a time partition of $[0, T]$ if and only if $\Pi_{T} = \left\{ (t_{i})_{1}^{n_T} \Big| 0=t_1<t_2<\cdots <t_{n_T}=T \right\}$.} of $[0, T]$, the corresponding output of XNODE is given by $\text{XNODE}_{\Theta}(\mathbf{x}, \Pi_{T}):=o_\mathbf{x}(\Pi_T)$. 

In essence, our proposed XNODE model is built upon the NODE model, though it incorporates an additional spatial variable to model the high dimensional latent process of the PDE solution $\tilde{u}$. The uniqueness of XNODE follows from that of NODE, by considering $\mathbf{x}$ as the additional model parameter. Similarly, the adjoint sensitivity of XNODE is inherited from NODE, allows us to compute gradients with low memory cost. We also establish the universal approximation property of XNODE in Theorem {\color{red} \ref{thm_universality}} in the Appendix.

To effectively use the initial condition of the PDE $u(0, \textbf{x}) = h(\textbf{x})$, the hidden neuron at the initial time $\mathbf{h}(0)$ is given by a neural network $\mathcal{N}^{\text{init}}$ of $h(\textbf{x})$ from $d$-dimensional input $\mathbf{x}$ to 1-dimensional input $h(x)$, without compromising any discriminate capacity of the model. As long as $\mathcal{L}_{\tilde{\theta}} \circ \mathcal{N}^{\text{init}}$ can approximate the identity map, which is guaranteed by the universality of the neural network, it is able to achieve the satisfactory fitting performance of initial condition.

\textbf{Point sampling}. At each iteration as in the WAN Algorithm \ref{WAN_alg}, to compute $\tilde L_{\text{int}}, \tilde L_{\text{init}}$ and $\tilde L_{\text{bdry}}$ defined in \eqref{L_int}, one needs to sample points $\{(t_i, \mathbf{x}_i)\}_i$ uniformly on the interior and the boundary of the domain $ [0,T] \times \Omega$ respectively and estimate the solution $u$ evaluated at the grid. Recall that the output of the $\text{XNODE}_{\Theta}(\mathbf{x}, \Pi_{T})$ is to approximate the solution $u$ evaluated at the points $\Pi_{T} \times \{\mathbf{x}\}$. Therefore,
unlike the sampling strategy of WAN, XNODE only requires sampling independently and uniformly on both time domain and spatial domain. Let us denote the sampled time partition and collocation points of the spatial domain $\Omega$ and domain boundary $\partial \Omega$ by $\Pi_T$, $S^{r}$ and $S^b$ respectively, with $S^r:=\{\mathbf{x}_j^r\in \Omega\}_{j=1}^{N_r}$ and $S^b:=\{\mathbf{x}_b^r\in \partial\Omega\}_{j=1}^{N_b}$. Note the proposed sampling method (e.g. see the solid red points 
in Figure \ref{fig:path-a} for an illustration) implicitly generates a uniformly sampled grid of $[0,T] \times \Omega$ in a more economical way by recycling the uniformly sampled spatial points at each sampled time.\footnote{
The collocation of points on the interior (resp. boundary) of the domain 
 $[0,T] \times \Omega$ (resp. $[0,T] \times \partial\Omega$) is given by 
 $\Pi_{T} \times S^{r}$ 
 (resp. $\Pi_{T} \times S^{b}$). To generate $N_{T} \times N_{r}$ points over the interior of the domain $[0,T] \times \Omega$, the number of random points generated by the WAN is $N_{T} \times (N_{r} +N_b)$ while that of our point sampling strategy is $N_{T} + N_{r} + N_b$. } 

\textbf{Weak solution prediction.}  For any point $\mathbf{x} \in \Omega$, we can construct the corresponding \emph{constant spatial path} (i.e. $X(t)\equiv \mathbf{x}, \forall t \in [0, T]$, e.g. see a solid line in Figure \ref{fig:path-a}). The output $\text{XNODE}_\Theta (\mathbf{x}, \Pi_{T})$ then provides to a discrete approximation of the weak solution $u$ evaluated at $\{(\mathbf{x}, t_{i})\}_{t_i \in \Pi_{T}}$ along this constant spatial path. By repeating this procedure on each point $\mathbf{x}^r_{i} \in S^{r}$, we can generate a collection of constant paths $(X^r_{i})_{i = 1}^{N_r}$ such that $X^r_{i} \equiv \mathbf{x}^r_i$. Similarly we generate a collection of constant paths $(X^b_{i})_{i = 1}^{N_b}$. In total, we obtain the corresponding 
$\left(\text{XNODE}_\Theta (X_i, \Pi_{T}) \right)_{\mathbf{x}_i \in S^{r} \cup S^b}$ to approximate the solution $u$ over the grid $\Pi_T \times (S^{r} \cup S^b)$. The algorithm of using NODE model to approximate the weak solution network is given in Algorithm \ref{NODE-u}.

\begin{algorithm}[H]
\caption{XNODE Network for modelling the weak solution $u$ on a constant domain $[0,T] \times \Omega$}  \label{NODE-u}
\textbf{Input:} sampled sets of points $S^{r} = \{ \mathbf{x}_{j}^r \in \Omega \}_{j = 1}^{N_r}$ and $S^{b} = \{ \mathbf{x}_{j}^b \in \partial \Omega \}_{j=1}^{N_b}$; $\Pi_{T} = (t_{i})_{i =0}^{n_T}$ time partition of $[0, T]$; XNODE parameter set $\Theta = (\theta_1, \theta_2, \tilde{\theta})$ \\ 
\textbf{Output:} $O$ \\
\algorithmiccomment{approximate $u$ at $S^r \times \Pi_{T}$ and $S^b \times \Pi_{T}$} \\
\textbf{Algorithm:}
\begin{algorithmic}[1]
\STATE initialize $O$ as an empty vector
\FOR{$\mathbf{x}$ in $S^{r} \cup S^{b}$}
\STATE compute $o_\mathbf{x} = \text{XNODE}_{\Theta}(\mathbf{x}, \Pi_{T})$;\\ 
\algorithmiccomment{approximate $u$  at $\{x\} \times \Pi_{T}$}
\STATE set $O = \text{concatenate}(O, o_\mathbf{x})$;
\ENDFOR
\RETURN output $O$ 
\end{algorithmic}
\end{algorithm}

\textbf{XNODE-WAN Algorithm}: Finally, we obtain the XNODE-WAN model by replacing the DNN network by the XNODE network for the primal solution in the WAN algorithm (see  Algorithm \ref{WAN_alg}).
In particular, we adapt the WAN algorithm in the following steps:
\begin{itemize}
\item
In line 1, the DNN network for the primary solution $u_\theta$ is replaced by the XNODE network.
\item
In line 2, the point sampling strategy is updated as described before. For each iteration, our algorithm implements such sampling to generate a set of random collocation data points over the domain. 
\item
In line 6 and line 12, the adjoint method is used to compute the derivative of XNODE model instead of the gradient descent method of DNN network in WAN. 
\end{itemize}

The full XNODE-WAN algorithm is outlined in Algorithm \ref{XNODE-WAN_indep_domain} in Appendix.

\subsubsection{Time-dependent domain}\label{subsec:NODE-WAN-b}
In this subsection, we consider a more general case $\mathcal{D} \subset  [0,T] \times \mathbb{R}^{d}$ which allows that the spatial domain is time-dependent. Let $\Omega(t)$ denote the spatial domain of $\mathcal{D}$ restricted to any $t \in [0, T]$, i.e. $\Omega(t) = \{\mathbf{x} | (t, \mathbf{x}) \in \mathcal{D}\}$.  
When $\Omega(t)$ changes over $t \in [0, T]$, we call $\mathcal{D}$ a time-varying domain. The proposed XNODE model in Section \ref{sec:independent} can not be directly applied to a time varying domain as the constant path $X \equiv \mathbf{x} \in \Omega (0)$ might go outside the domain $\mathcal{D}$ at some time $s \in [0, T]$, i.e. $(s, \mathbf{x}) \notin \mathcal{D}$ (see e.g., the sample point marked in $\circ$ in Fig. \ref{fig:path-b}).

To circumvent the problem, we shall divide each constant path $X \equiv \mathbf{x}$ into multiple sub-paths and only keep the sub-paths remaining within the domain $\mathcal{D}$. To be more specific, we firstly define the entry and exit points on a constant path $X$.

\begin{definition}\label{def:entry}
Let $\mathcal{D}$ be a bounded domain in $[0,T] \times \mathbb{R}^d $ and fix $\mathbf{x} \in \mathbb{R}^{d}$. Let $X$ denote a constant path taking value in $\mathbf{x}$, i.e. $X(t) \equiv \mathbf{x}, \forall t \in [0, T]$. For $i = 1, 2, \cdots$, the $i^{th}$ entry point and exit point of the constant path $X$, denoted by $\underline{\tau}^{(i)}_X$ and $ \overline{\tau}^{(i)}_X$ are defined as follows:
\begin{eqnarray*}
&&\underline{\tau}^{(1)}_X = \inf_{t \in [0, T]} \{t | (t, \mathbf{x}) \in \mathcal{D}\};\quad\\
&&
\overline{\tau}^{(i)}_X = \inf_{t \in [0, T]} \{t | (t, \mathbf{x}) \notin \mathcal{D}, t \geq \underline{\tau}^{(i)}_X \}; \quad \underline{\tau}^{(i+1)}_X = \inf_{t \in [0, T]} \{t | (t, \mathbf{x}) \in \mathcal{D}, t \geq \overline{\tau}^{(i)}_X \}.
\end{eqnarray*}
where $\inf \emptyset = + \infty$ by convention. Let $N_\mathbf{x}$ denote the total number of finite $\overline{\tau}^{(i)}_{X}$, i.e. 
\begin{eqnarray*}
N_\mathbf{x} = \sup_{i \in \{1, 2, \cdots\} }\{i | \overline{\tau}^{(i)}_{X} < +\infty \},
\end{eqnarray*}
where $\sup \emptyset = -\infty$ by convention.
\end{definition}
An illustration example is given in Fig. \ref{fig:illustration}, with the blue solid dots along the blue lines indicating where the entry and exit points locate for the given constant path $X$.

For the ease of notation, when there is no ambiguity, we omit the subscript $X$ in $\underline{\tau}^{(i)}_X$ and $ \overline{\tau}^{(i)}_X$ for the rest of the paper. When restricting the constant path $X\equiv \mathbf{x}$ to the time interval from the $i^{th}$ entry point $\underline{\tau}^{(i)}$ to the $i^{th}$ exit point $\overline{\tau}^{(i)}$, we obtain a constant sub-path denoted by $ X^{i} = X_{[\underline{\tau}^{(i)}, \overline{\tau}^{(i)}]}$, whose time augmented path \footnote{Let $f: [0, T] \rightarrow \mathbb{R}^{d}$. The time augumented path of $f$ is defined as a map from $[0, T] \rightarrow [0, T] \times \mathbb{R}^{d}: t \mapsto (t, f(t))$.} takes value in the closure of $\mathcal{D}$ (e.g. see the solid red lines in Fig. \ref{fig:path-b}).

We assume that for the given domain $\mathcal{D}$, the entry and exit points of every constant path $X\equiv \mathbf{x}$ can be computed and $N_\mathbf{x}$ is finite. For any given time partition $\Pi_T$ and the spatial point $\mathbf{x}$, we propose the following way to assign the collocation of time points from $\Pi_T$ to each sub-path $X^{i}$ for the XNODE model. We define a set $\Pi^{\mathbf{x}, i}_{T}$, which is composed with the $i^th$ entry starting point $\underline{\tau}^{(i)}$ and all the elements of $\Pi_{T}$ belonging to the $i^{th}$ time interval $[\underline{\tau}^{(i)}, \overline{\tau}^{(i)}]$ for $i \in \{1, \cdots, N_x\}$. Hence $\Pi^{\mathbf{x}, i}_{T}$ is the time collocation points associated with the sub-path $X^{i}$. Fig. \ref{fig:illustration} presents an example of computing the entry and exit points and constructing time collection points of each constant subpaths.

\begin{example}
In Fig. \ref{fig:illustration},  a shaded grey region represents the domain $\mathcal{D}$. For the given constant path $X = \mathbf{x}$ marked as the solid blue line, one can compute two pairs of finite entry and exit points, denoted by $\{(\underline{\tau}^{(i)}, \overline{\tau}^{(i)})\}_{i = 1}^{2}$, which corresponds to two constant sub-paths within the domain $\mathcal{D}$. Let $\Pi_{T}=  (t_{i})_{i = 1}^{8}$ denote a time partition of length $8$ in Figure \ref{fig:illustration}. The time collocation points of the first and second sub-paths are $\Pi^{\mathbf{x},1}_{T} = (\underline{\tau}_1, t_2, t_3)$ and 
 $\Pi^{\mathbf{x}, 2}_{T} = (\underline{\tau}_X^{(2)}, t_{6}, t_{7}, t_{8}, \overline{\tau}_X^{(2)})$ respectively\footnote{Note $t_1 = \underline{\tau}_{X}^{(1)}$ and hence $\Pi^{\mathbf{x},1}_{T} = (\underline{\tau}^{(1)}) \cup (t_1, t_2, t_3) =  (\underline{\tau}^{(1)}, t_2, t_3)$.}.
\end{example}

\begin{figure}
    \centering
    \includegraphics[width = 0.6\textwidth]{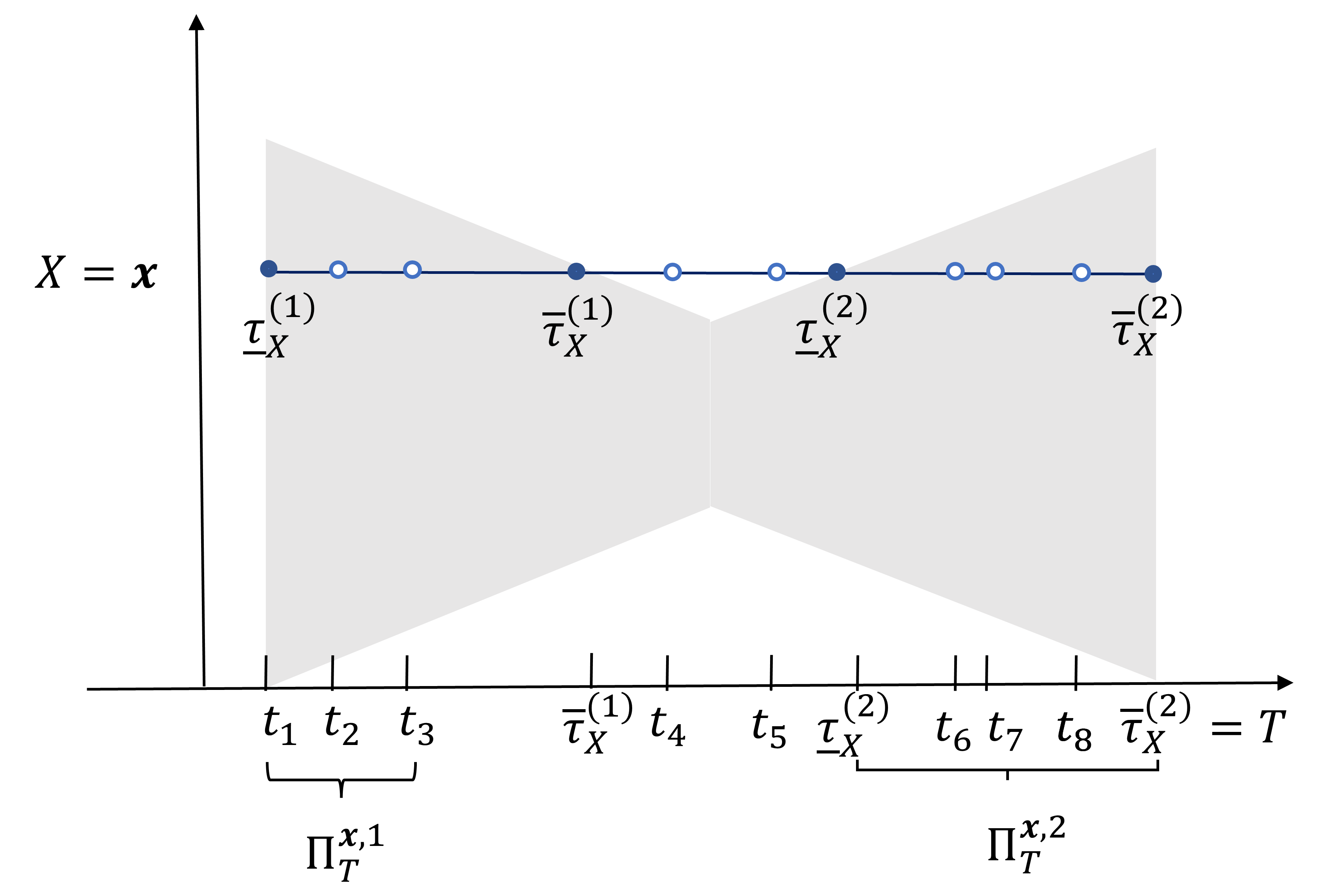}
    \caption{An illustration for entry and exit points of a constant path (see Definition \ref{def:entry}) and the corresponding subpaths $\Pi^{\mathbf{x}, 1}_{T}$ and $\Pi^{\mathbf{x}, 2}_{T}$ for the given domain presented as a shaded region. The solid line presents the constant path $X \equiv \mathbf{x}$. The time coordinates of the blue solid dots are the entry and exit points, while the blue circles represent $\mathbf{x} \times \Pi_T$ the samples along the path. Clearly there are two points $(t_4, \mathbf{x})$ and $(t_5, \mathbf{x})$ along the path but not in the shaded domain, which are filtered out when constructing sub-paths within the domain.}
    \label{fig:illustration}
\end{figure}
Now we are ready to outline the proposed modification of Algorithm \ref{NODE-u} for the time-varying domain as follows.

\textbf{Point sampling}: First of all, we find the the minimum time-independent domain to cover $\mathcal{D}$, i.e.  $\tilde{D} = [0, T] \times \Omega_{\max}$,
where $\Omega_{\max} = \cup_{t \in [0, T]}\Omega(t)$. We adopt the point sampling method in Section \ref{sec:independent} to generate the uniform grid to cover $[0, T] \times \Omega_{\max}$, i.e. $\Pi_{T} \times S^{r}$, where $S^{r}$ is sampled uniformly from $\Omega_{\max}$. Secondly, for every $\mathbf{x}\in S^r$, we compute $\{(\underline{\tau}^{(i)}, \overline{\tau}^{(i)})\}_{i = 1}^{N_\mathbf{x}}$. We then construct $\mathbf{x}$-valued constant sub-paths and determine the time collocations points associated with each sub-path as described above. The algorithm of computing $\Pi^{\mathbf{x}, i}_{T}$ is given in Algorithm \ref{alg:sub-paths-construction}. To sample of the boundary points of the general domain $\mathcal{D}$, we use the same method of WAN \cite{zang2020weak}, i.e. uniformly sample a number of points from the boundary $\partial{\mathcal{D}}$.

 \begin{algorithm}[H]
\caption{Compute the constant sub-paths and the corresponding time collocation points}  \label{alg:sub-paths-construction}
\textbf{Input:} domain $\mathcal{D}$, $\mathbf{x} \in \mathbb{R}^{d}$, $\Pi_{T}$ \\
\textbf{Output:} $(\Pi_{T}^{\mathbf{x}, i})_{ i = 1}^{N_\mathbf{x}}$\\
\textbf{Algorithm:}
\begin{algorithmic}[1]
\STATE compute the entry and exit points $(\underline{\tau}^{(i)}, \overline{\tau}^{(i)})_{i = 1}^{N_\mathbf{x}}$
\FOR{i from 1 to $N_\mathbf{x}$}
\STATE initialize $\Pi^{\mathbf{x},i}_T$ as an empty list;
\STATE add $\underline{\tau}^{(i)}$ to $\Pi^{\mathbf{x},i}_T$;
\FOR{$t$ in $\Pi_T$ and $t$ in $(\underline{\tau}^{(i)}, \overline{\tau}^{(i)}]$}
\STATE add t to $\Pi^{\mathbf{x},i}_T$;
\ENDFOR
\ENDFOR
\end{algorithmic}
\end{algorithm}

\textbf{Weak solution prediction} For each $\mathbf{x} \in S^{r}$, we have $N_\mathbf{x}$ many constant sub-paths defined on the sub-time interval $[\underline{\tau}^{i}, \overline{\tau}^{i}]$ with corresponding discrete time collocation points $\Pi^{\mathbf{x},i}_{T}$. Therefore we can estimate the weak solution $u$ evaluated at $\{\mathbf{x}\} \times \Pi^T_{\mathbf{x}, i}$ by $\text{XNODE}(\mathbf{x}, \Pi_T^{\mathbf{x}, i})$ for $i \in \{1, \cdots N_\mathbf{x}\}$. Note that the initial condition of XNODE model for each constant sub-path $X^{i}$ is well defined as $u(\underline{\tau}_{i}, \mathbf{x})$ is given by either the initial condition or boundary condition, i.e.
\begin{eqnarray}\label{Time_varying_eqn}
\mathbf{h}(\underline{\tau}^{(i)}) = \begin{cases}
\mathcal{N}^{\text{init}}_{\theta_1}(h(\mathbf{x})), & \text{if }i = 0 \text{ and }\underline{\tau}^{(i)} = 0;\\
\mathcal{N}^{\text{init}}_{\theta_1}(g(\mathbf{x},\underline{\tau}^{(i)})), & {\text{otherwise.}}
\end{cases} 
\end{eqnarray}

We now present the modified algorithm of XNODE for the time-dependent domain in Algorithm \ref{NODE-timevary}, on top of Algorithm \ref{alg:sub-paths-construction}. 
 \begin{algorithm}[H]
\caption{XNODE Network for modeling the weak solution $u$ on a time-dependent domain $\mathcal{D}$}  \label{NODE-timevary}
\textbf{Input:} sampled sets of points $S^{r} = \{ \mathbf{x}_{j}^r \in \Omega_{\max} \}_{j = 1}^{N_r}$  and $\Pi_{T} = (t_{i})_{i =0}^{n_T}$ time partition of $[0, T]$; XNODE parameter set $\Theta = (\theta_1, \theta_2, \tilde{\theta})$ \\ 
\textbf{Output:} $O$ \algorithmiccomment{approximate $u$ at $S^r \times \Pi_{T}$} \\
\textbf{Algorithm:}
\begin{algorithmic}[1]
\STATE initialize $O$ as an empty vector
\FOR{$\mathbf{x}$ in $S^{r} $}
\STATE compute $(\Pi_{T}^{\mathbf{x}, i})_{ i = 1}^{N_\mathbf{x}}$ with the input arguments $\mathcal{D}$, $x$ and $\Pi_{T}$ through Algorithm \ref{alg:sub-paths-construction};
\STATE initialize $o_\mathbf{x}$ as a zero vector of length $N_\mathbf{x}$;
	\FOR{$i$ from $1$ to $N_x$}
\STATE compute $o_{\mathbf{x},i}= \text{XNODE}_{\Theta}(\mathbf{x},\Pi_{\mathbf{x},i})$;\\ 
\algorithmiccomment{approximate $u$  at $\{\mathbf{x}\} \times \Pi_{\mathbf{x},i}^T$}
\STATE add $o_{\mathbf{x},i}$ to $o_{\mathbf{x}}$ in the corresponding locations based on $ \Pi_{\mathbf{x},i}^T$;
\ENDFOR
\STATE set $O = \text{concatenate}(O, o_\mathbf{x})$;
\ENDFOR
\RETURN output $O$
\end{algorithmic}
\end{algorithm}

\section{Numerical Results}\label{sec:numerical}

In this section, we conduct several numerical experiments of parabolic PDEs defined on time independent domains or time varying domains, respectively. We compare the performance of the proposed XNODE-WAN algorithms (Algorithm \ref{NODE-u} and Algorithm \ref{NODE-timevary}) with the baseline WAN algorithm \cite{zang2020weak}. 
In particular, following the numerical example in \cite{zang2020weak}, we consider PDE examples in the form of a $d$-dimensional nonlinear diffusion-reaction equation (Eq. \eqref{eqn:problem}) defined on a bounded domain $\mathcal{D} \subset  [0,1]\times \mathbb{R}^{d}$:
\begin{equation}\label{eqn:problem}
\begin{cases}
\partial_t u-\bigtriangleup u - u^2 -f=0 & \text{for }(t,\mathbf{x})\in \mathcal{D}\\
u-g=0 & \text{on }\partial\mathcal{D}\\
u(0,\mathbf{x})-h(\mathbf{x})=0 & \text{on }\Omega(0),
\end{cases}
\end{equation}
where $f$, $g$ and $h$ may vary from example to example. In Section \ref{sec:num_timenovarying}, we consider PDE \eqref{eqn:problem} on two time independent domains $\mathcal{D}= [0,1]\times \Omega $ where $\Omega $ is a hyper cube or a unit ball. 
Note our example where $\Omega$ is a hyper cube is the same as that in Section 4.2.6 of \cite{zang2020weak}. We then investigate the scalability of our proposed method with respect to high spatial dimension in Section \ref{sec:scalability} and follow by a sensitivity analysis in Section \ref{sec:sensitivity}.  
At the end of this section, we provide an example of a time-varying domain. Numerical implementation of our work is available at \url{https://github.com/paulvoliva/XNODE-WAN-PDE-solver.git}.

\subsection{Experimental setup}\label{sec:setup}
To quantitatively access the accuracy and efficiency of the method, we consider the following test metrics:
\begin{enumerate}
    \item
    The relative error:  $||u_\theta - u||_{L^2}/||u||_{L^2}$, where
$u_\theta$ and $u$ are the approximation and exact solution of the problem with  $||u||_{L^2} =\left(\int_{0}^{T} 
\int_{\Omega(t)}|u(t,\mathbf{x})|^2 \mathrm{d}x\mathrm{d}t\right)^{\frac{1}{2}}.$
\item Time per epoch: the average time elapsed in each epoch in the training process.  
\item $N_{\epsilon}$: for any error tolerance level $\epsilon >0$, the minimum number of epochs such that the relative error of the trained model is no more than $\epsilon$. 
    \item $T_{\epsilon}$: given any error tolerance level $\epsilon >0$, the minimum of time (seconds) such that the relative error of the trained model is no more than $\epsilon$; we would also refer to it the time at convergence. 
\end{enumerate}

For a fair comparison, we choose the following hyperparametes, which are the same for both WAN and XNODE-WAN: $N_r= N_b = 400$, $n_T=20$, $K_u=2$, $K_\phi=1$, $\alpha=\gamma=400,000\times d^2$ and the learning rate of the adversarial network $\tau_\eta=0.04$ (see Table \ref{table2} for the summary of notations of variables used in models and algorithms). Note that we set $N_\mathcal{D}=N_r*n_T$ and $N_{\partial\mathcal{D}}=N_b*n_T$ for WAN to have a fair comparison. Besides we keep the same network architecture of the adversarial network (test function $\phi_{\eta}$) the same for both WAN and XNODE-WAN. However, as the training performance of each method may be sensitive to the learning rate, the learning rate of the primal network should be adjusted to the XNODE or DNN separately. By hyper-parameter tuning, the learning rate of the XNODE-WAN and WAN for the primal network are chosen to be 0.015 and 0.00005 respectively. Note when increasing the learning rate of the WAN algorithm from 0.00005, the WAN may suffer from either divergence or slower training. One may refer the effects of the learning rate in Section \ref{sec:sensitivity}. We use the above hyper-parameter setting for all the following numerical examples.

\subsection{Time-independent domains}\label{sec:num_timenovarying}

In this subsection, we investigate the performance of the proposed method on a time-independent domain $\mathcal{D}=[0,1]\times \Omega$. To demonstrate that our method is applicable to a general spatial domain $\Omega$ , we consider two examples of $\Omega$, i.e. (1) a hyper-cube $\Omega = [0, 1]^{d}$ and (2) a unit ball $\Omega = \{x \in \mathbb{R}^d| |x| \leq 1\}$ for $d=5$.

\begin{example}\label{ex1}
We consider the same example in Section 4.2.6 in \cite{zang2020weak} with $f, g, h$ specified as follows:
\begin{eqnarray}\label{eqn:fgh_compare}
\begin{cases}
f(t,\mathbf{x}) 
= (\pi^2 - 2)\sin\left(\frac{\pi}{2}\right)
\cos\left(\frac{\pi}{2} e^{-t}\right) 
    -4 \sin^2(\frac{\pi}{2}x_1) \cos(\frac{\pi}{2}x_2)e^{-2t};&\\
g(t,\mathbf{x}) 
= 2\sin(\frac{\pi}{2}x_{1})\cos(\frac{\pi}{2}x_{2})e^{-t};&\\
h(\mathbf{x}) \quad
= 2 \sin(\frac{\pi}{2}x_1)\cos(\frac{\pi}{2}x_2).&
\end{cases}
\end{eqnarray} 
where $\Omega=[0,1]^5$. 
The exact solution is given by 
\begin{equation}\label{eqn:eg1_sol}
u(t,\mathbf{x}) = 2\sin\left(\frac{\pi}{2}x_1\right) \cos\left(\frac{\pi}{2}x_{2}\right)e^{-t}.
\end{equation}
\end{example}

The stopping criteria for both algorithms is set to achieve a relative error tolerance level $\epsilon = 10^{-3}$.  
 We also randomly test $50$ trials on randomly selected data to measure the relative $L^2$ error on the obtained models.  We report the mean and the standard deviation of relative error of the XNODE-WAN and WAN in Table \ref{table}, which demonstrates that XNODE-WAN model significantly outperforms in terms of computational time  and iteration number. In particular,
 to reach the same stopping criteria, XNODE-WAN model saves $99.8\%$ in time and only takes less than 1.5 minutes for training in this case. Furthermore, XNODE-WAN and WAN  need $165$ and $15,463$ iterations respectively whereas the time used for each iteration is comparable. In addition, \Cref{fig: sol_training_loss1} gives the evolution of relative errors over time for both models. Clearly the proposed model converges much faster than that of WAN.

\begin{table}[H]
\centering
\resizebox{\columnwidth}{!}{
\begin{tabular}{|l|l|l|l|l|l|}
\hline
         & Relative error ($\pm$ SE)  & Time per epoch (s) & $N_\varepsilon$ &  $T_\varepsilon $ (s)\\
         \hline
WAN      &  2.6\%  $\pm$ 0.0144\%  &  0.38 & 15,463 & 5865 \\
\hline
XNODE-WAN &  1.7\% $\pm$ 0.0114\% & 0.35 &  211  &  74 \\
\hline
\end{tabular}}
\caption{Performance comparison between XNODE-WAN and WAN on \Cref{ex1}.}\label{Tab1_eg1}
\label{table}
\end{table}

\Cref{fig: sol_training_loss2} and \Cref{fig: sol_training_loss4} demonstrates the comparison between the XNODE-WAN and WAN estimator for the PDE solution $u(T, \mathbf{x})$ at the convergence and their respective absolute point-wise error on the $(x_1, x_2)$ domain. It shows that the out-of-sample point-wise absolution error of our method has less magnitude than that of WAN. The superior performance of our method over WAN is also verified by a smaller relative error ($1.7\% \pm 0.0114\%$ v.s. $2.6\% \pm 0.0144\%$) with statistical significance.  

Finally, in \Cref{fig: sol_training_loss5}, we show that when XNODE-WAN reaches the stopping criteria at 74s, the WAN solution is far from being close to the true solution with large error.

\begin{figure}[!hb]
   \begin{subfigure}[b]{0.34\textwidth}
\includegraphics[width=\textwidth]{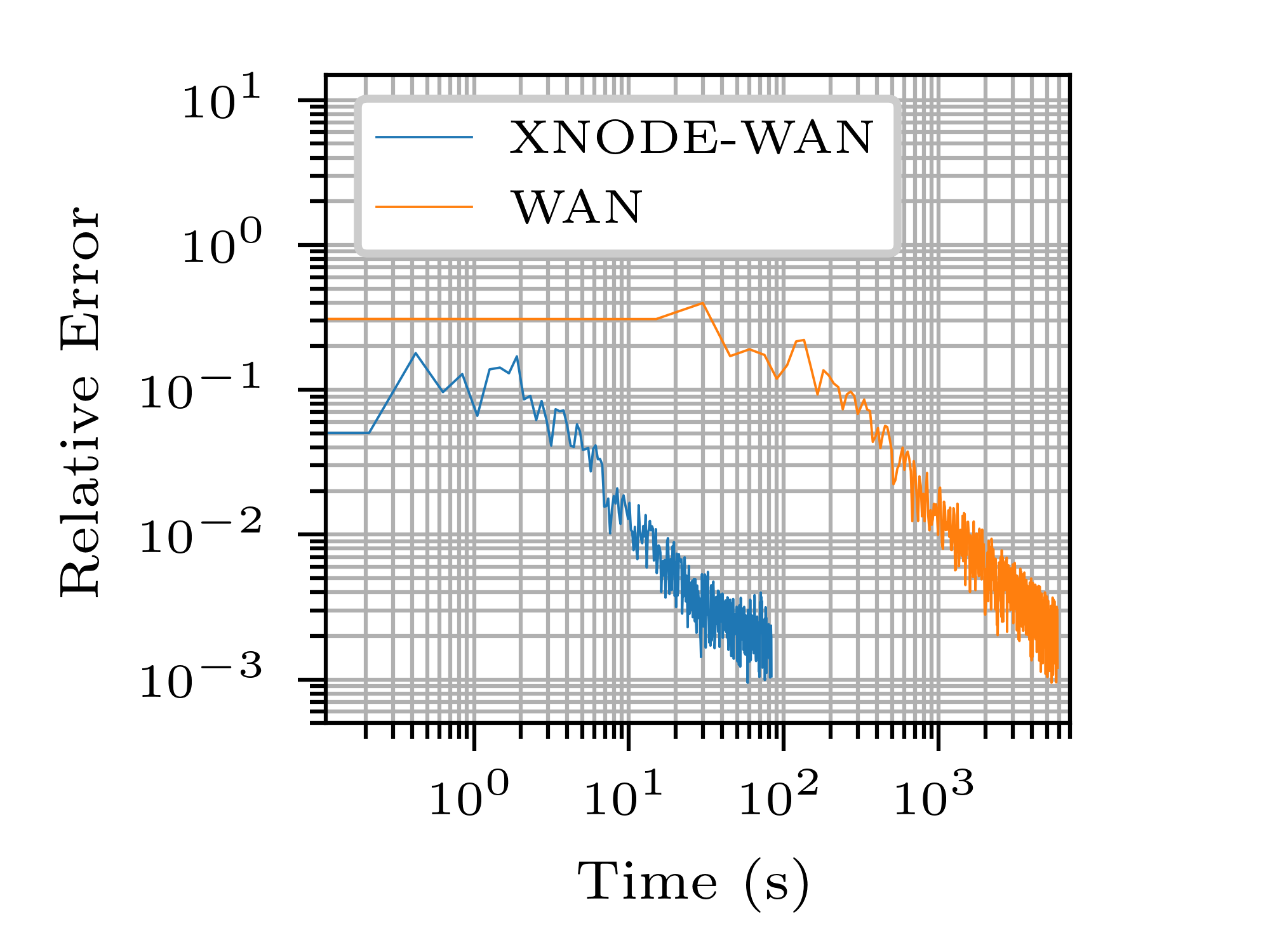}
\subcaption{The relative error over time}
\label{fig: sol_training_loss1}
   \end{subfigure}
     \begin{subfigure}[b]{0.65\textwidth}
\includegraphics[width=\textwidth]{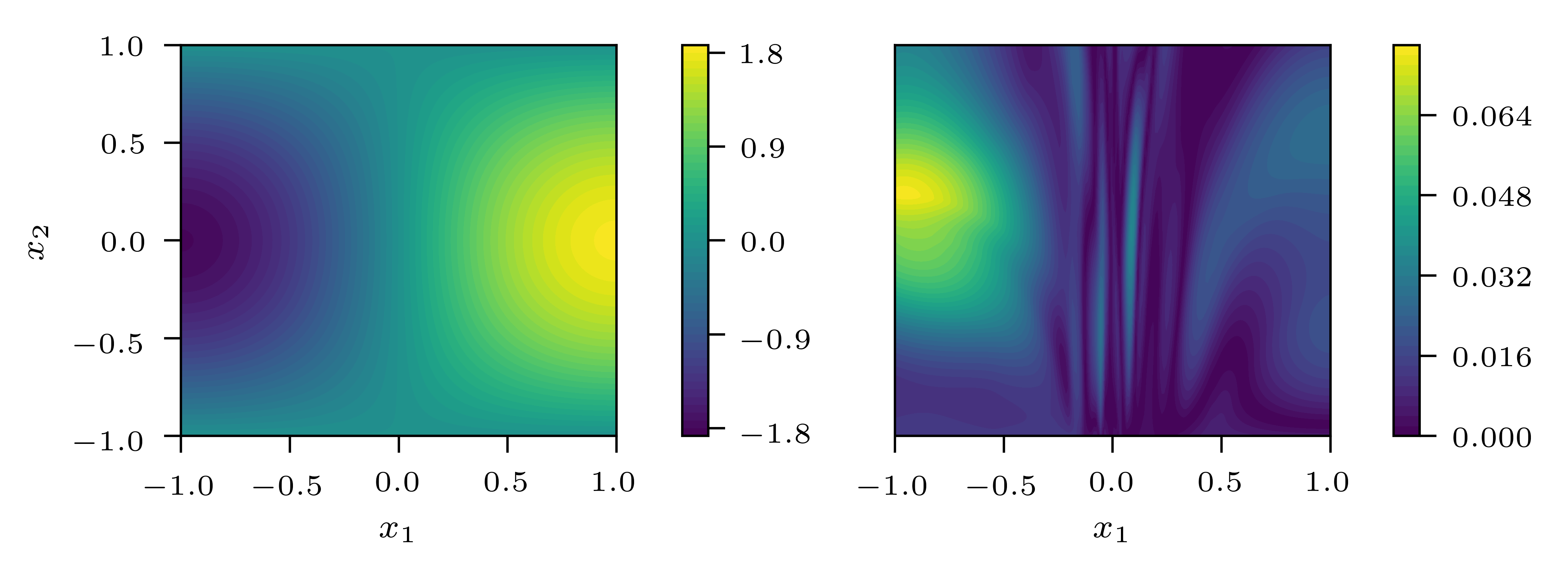}
\subcaption{XNODE-WAN solution and its error at convergence (74s)}\label{fig: sol_training_loss2}
   \end{subfigure} \\
     \begin{subfigure}[b]{0.34\textwidth}
\includegraphics[width=\textwidth]{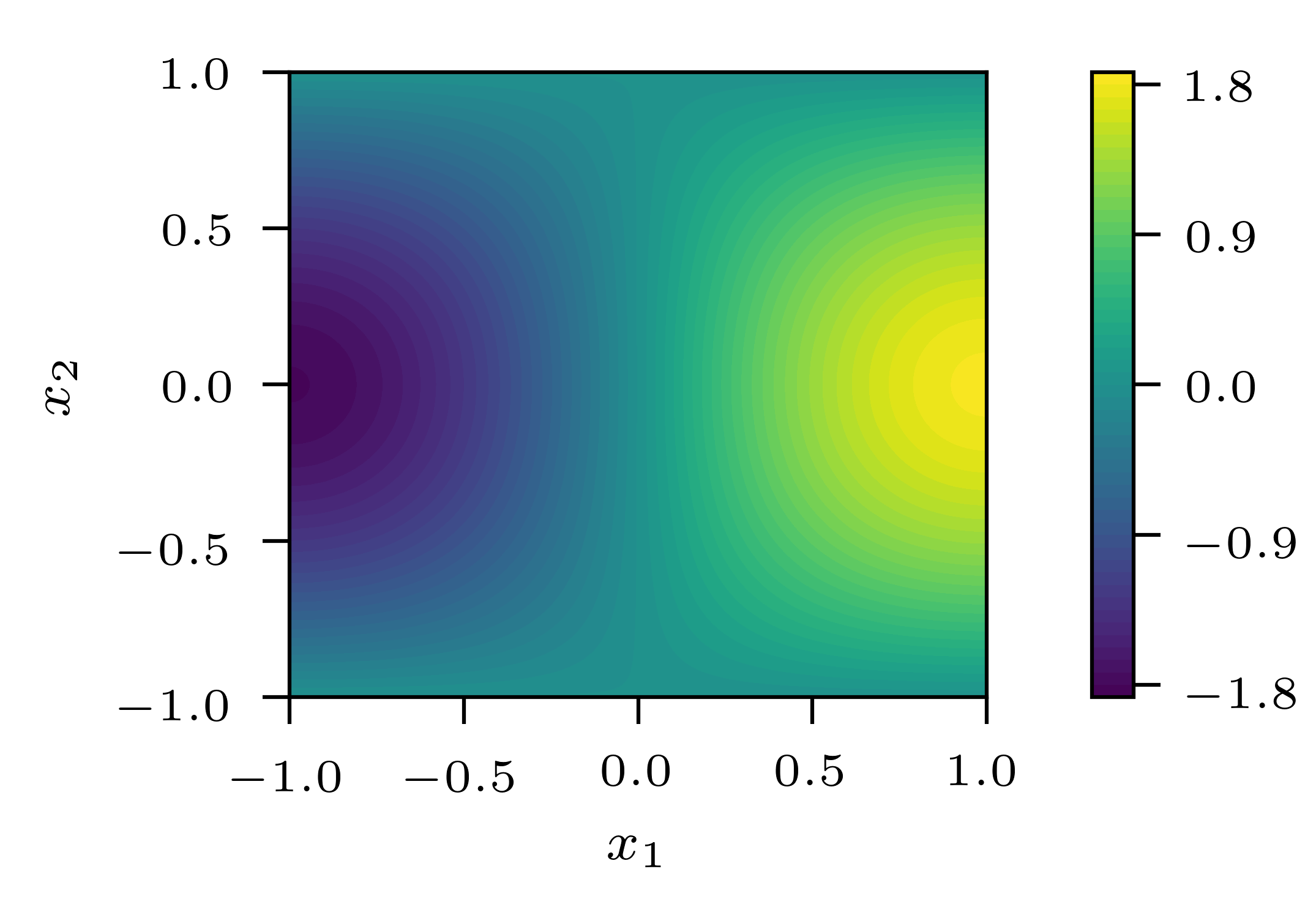}
\subcaption{True solution $u(T, \mathbf{x})$}\label{fig: sol_training_loss3}
   \end{subfigure}
     \begin{subfigure}[b]{0.65\textwidth}
\includegraphics[width=\textwidth]{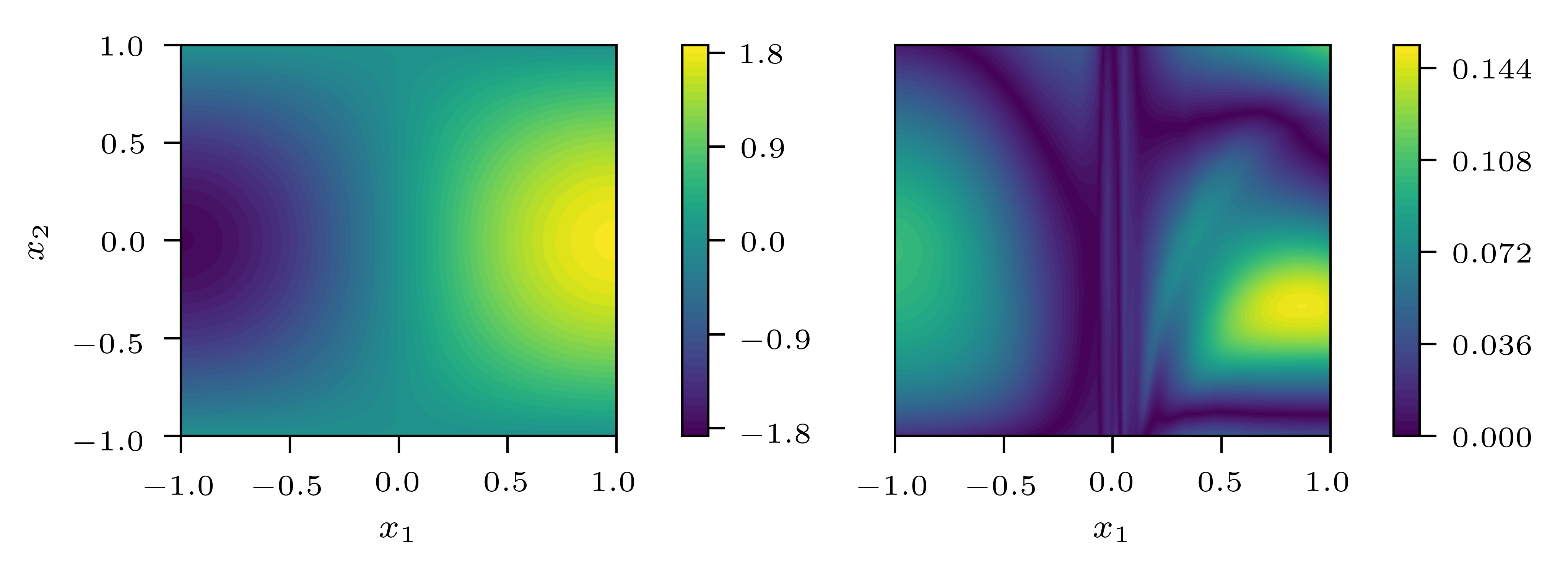}
\subcaption{WAN solution and its error at convergence (5865s)}\label{fig: sol_training_loss4}
   \end{subfigure}\\   
      \begin{subfigure}[b]{0.31\textwidth}
      ~
      \end{subfigure}
       \quad 
       \begin{subfigure}[b]{0.65\textwidth}
\includegraphics[width=\textwidth]{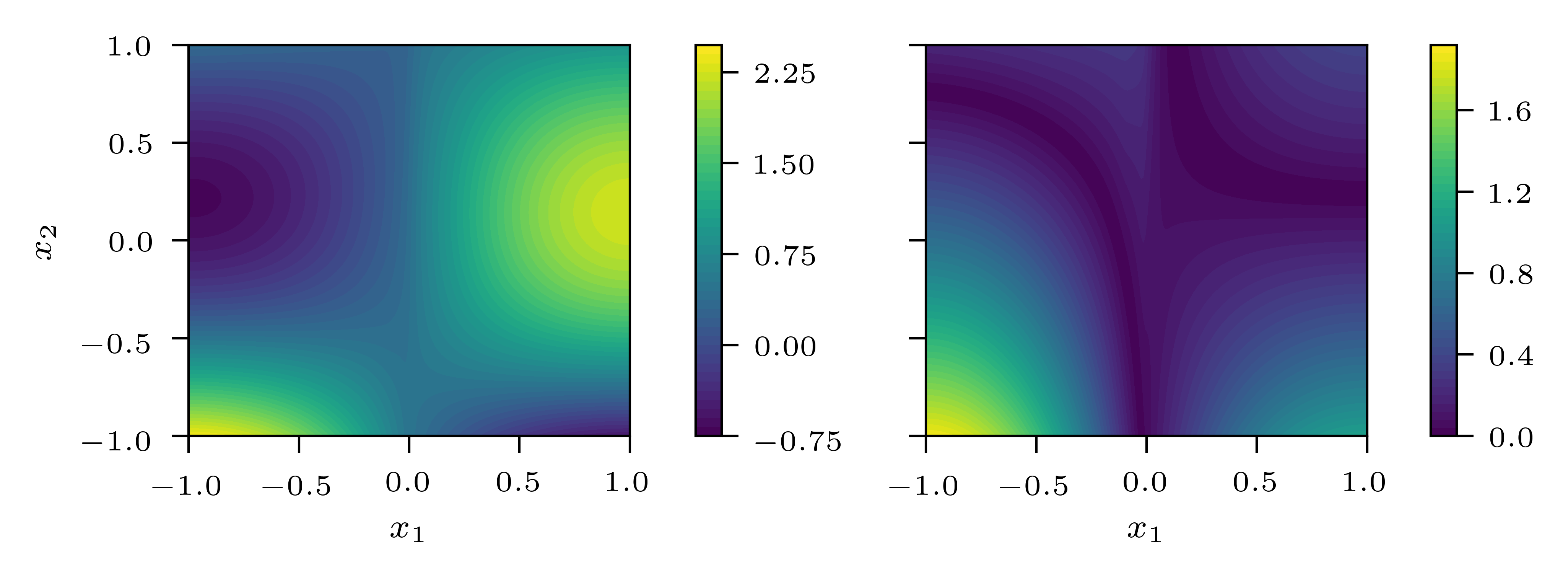}
\subcaption{WAN solution and its error at 74s}\label{fig: sol_training_loss5}
   \end{subfigure}
\caption{
Performance comparison between XNODE-WAN and WAN in \Cref{ex1}. In the right subplots of (b), (d) and (e), the error is the poitwise absolute error $|u(T,  \mathbf{x}) - \hat{u}(T, \mathbf{x})|$, where $\hat{u}$ is the model estimated solution. All the images (b)-(e) only show 2D slice of $x_3 = x_4 = x_5 = 0$ as the true solution $u$ of this example only depends on the first two coordinates. 
}\label{fig: sol_training_loss}
\end{figure}
\begin{example}\label{ex2}
We consider the same problem as in \Cref{ex1} but on a $5$-dimensional hypersphere domain, i.e., $\mathcal{D} = [0,1]\times B(\mathbf{x}_c, 0.5)$, where $\mathbf{x}_c=[0.5,0.5,0.5,0.5,0.5]^T$. The exact solution is given in \eqref{eqn:eg1_sol} on $\mathcal{D}$.  
\end{example}

We again implement both models and the results is shown in \Cref{fig:circle_ex} with the same stopping criteria, i.e., the relative $L^2$ error $\epsilon =10^{-3}$. The evolution of the relative $L^2$ error for the hypersphere domain is similar to that in \Cref{ex1}.
In particular, the XNODE-WAN model achieves the targeted error tolerance within 100s while it takes longer than 3,000s for WAN to reach the stopping criteria.

This example justifies the advantage of using XNODE-WAN for solving parabolic PDE with an irregular spatial domain $\Omega$.

\begin{figure}[h!]
    \begin{subfigure}[b]{0.35\textwidth}
    \includegraphics[width=\textwidth]{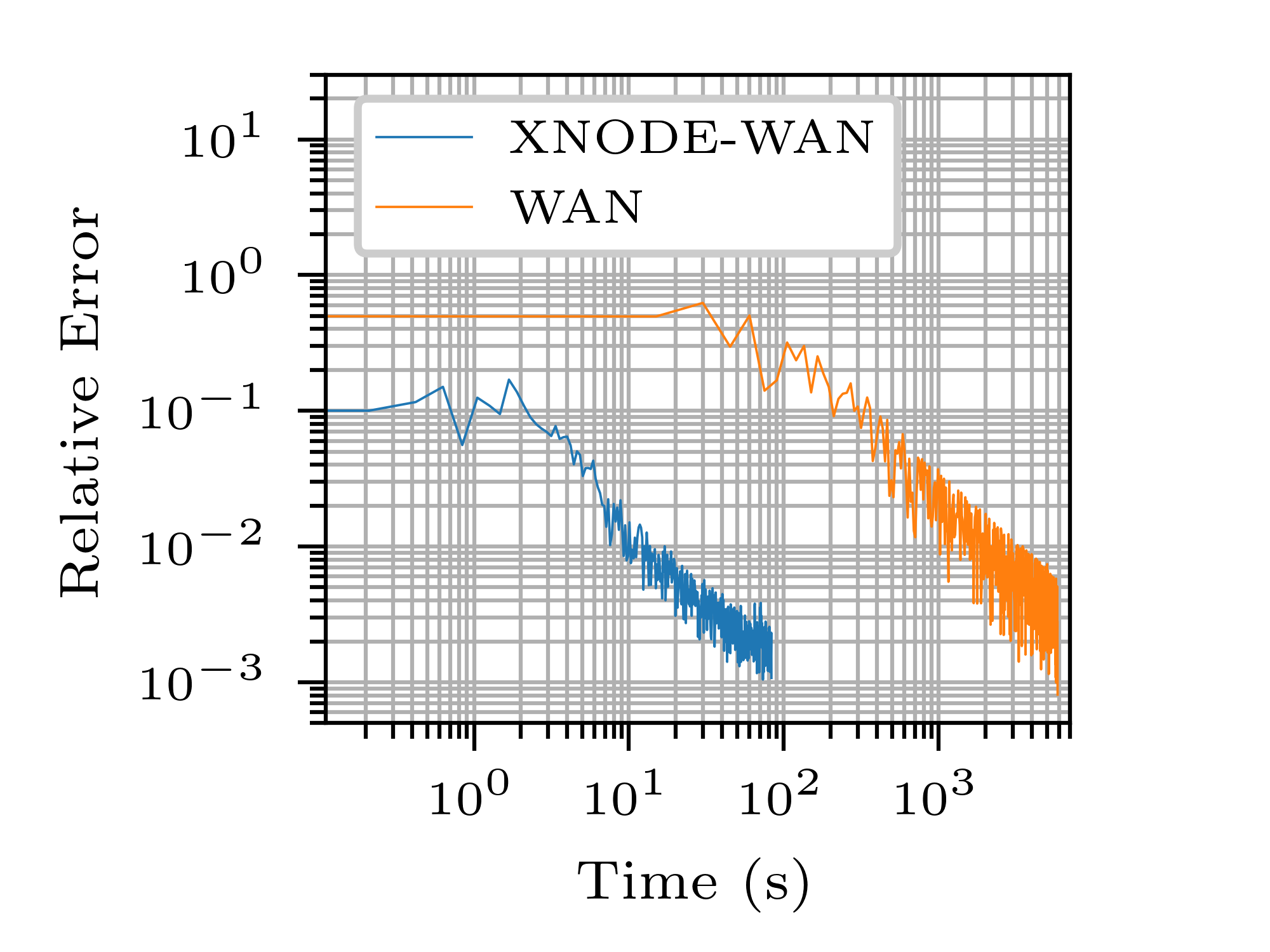}
    \subcaption{The relative error over time} \label{fig:circle_ex1}
    \end{subfigure}\hspace{0.0cm}
\begin{subfigure}[b]{0.64\textwidth}
    \includegraphics[width=\textwidth]{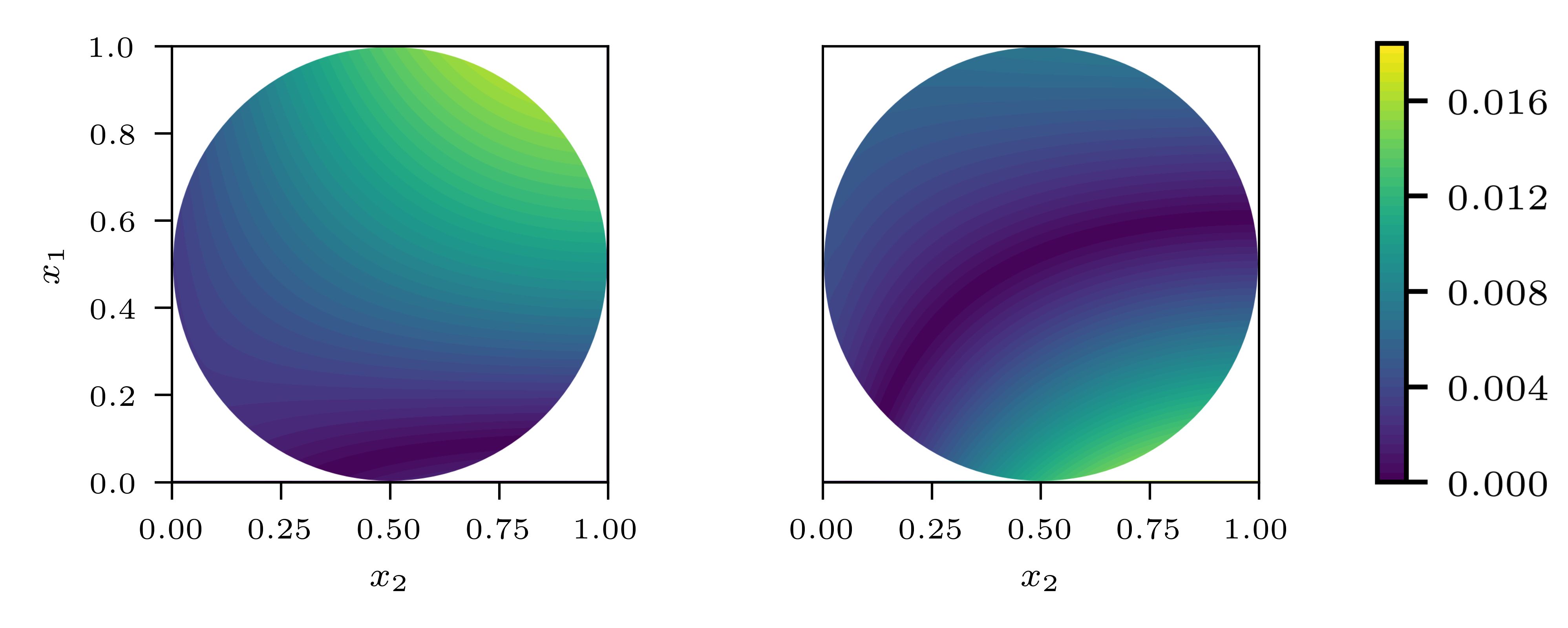}
    \subcaption{Error at convergence for XNODE-WAN (left) and WAN (right)} \label{fig:circle_ex2}
    \end{subfigure}
    \caption{
    Performance comparison of XNODE-WAN and WAN in \Cref{ex2}, where $\Omega$ is a $5$ dimensional unit ball. (a) the evolution of relative $L^2$ error of both XNODE-WAN and WAN method over time; (b) The plot of poitwise absolute error $|u(T,  \mathbf{x}) - \hat{u}(T,  \mathbf{x})|$, where $\hat{u}$ is the trained model of each method until hitting the stopping criteria. All the images only show 2D slice of $x_3 = x_4 = x_5 = 0$ as the true solution $u$ of this example only depends on the first two coordinates.    
    }
    \label{fig:circle_ex}
\end{figure}

\subsection{Scalability analysis}\label{sec:scalability}
In this subsection, we investigate the scalability of the proposed XNODE-WAN model with respect to the spatial dimension.
\begin{example}\label{Ex_sensitivity}
We consider Eq. \eqref{eqn:problem} on different spatial dimension $d \in \{4, 8, 16, 32, 64\}$, where $f, g$ are defined on $  [0, 1]\times\Omega$ and $h$ are defined on $\Omega$ as follows:
\begin{align}\label{eqn:fgh_scalability}
\begin{cases}
f(t,\mathbf{x}) = \frac{2e}{1-e}\left(\frac{\pi}{2}\right)^d\Big(\left(\frac{ \pi}{2} - 2\right) e^{-t} \Pi_{i=1}^d \sin\left(\frac{\pi}{2}x_i + \frac{\pi i}{2}\right) - 4 e^{-2t} \Pi_{i=1}^d \sin^2\left(\frac{\pi}{2}x_i + \frac{\pi i}{2}\right)\Big),&\\
g(t,\mathbf{x}) = \frac{2e}{1-e}\left(\frac{\pi}{2}\right)^d 2 e^{-t} \Pi_{i=1}^d \sin\left(\frac{\pi}{2}x_i + \frac{\pi i}{2}\right),&\\
h(\mathbf{x}) = \frac{2e}{1-e}\left(\frac{\pi}{2}\right)^d 2 \Pi_{i=1}^d \sin\left(\frac{\pi}{2}x_i + \frac{\pi i}{2}\right).&
\end{cases}\end{align}
It is easy to verify that the solution to \eqref{eqn:problem} is given by
$$u(t,\mathbf{x}) = \left(\frac{\pi}{2}\right)^d 2 e^{-t} \Pi_{i=1}^d \sin\left(\frac{\pi}{2}x_i + \frac{\pi i}{2}\right).
$$ 
\end{example}

Note in \cref{Ex_sensitivity}, $||u||_{L^2}=1$ regardless of the problem dimension $d$. For both methods, we choose the number of points $N_r, N_b = 800d$ and $\epsilon = 10^{-2}$. 

We investigate the performance of our proposed XNODE-WAN method with different spatial dimension $d$ in terms of the relative error evolution over the training, total training time and training time per epoch, which is visualized in \Cref{fig:scalability}.  \Cref{fig:scalability1} shows that even for a high dimensional $d = 64$, our method is able to converge to the relative error tolerance $\epsilon = 10^{-2}$ within $10^{2}$ seconds. %

Shown in \Cref{fig:scalability2}, for a dimension as high as $64$, to acheive the error tolerance $10^{-2}$, the training time of WAN is 4,000s, roughly 40 times of that of XNODE-WAN. It appears that, as $N_r$ and $N_b$ are chosen proportional to $d$, the average training time increases approximately linearly in dimension $d$ for XNODE-WAN. Clearly for the same dimension $d$, the training time per epoch from WAN is longer than that from XNODE-WAN, especially for large values of $d$. In total, \Cref{fig:scalability} show that XNODE-WAN has greater potential in scalability for higher dimensional PDEs. 

\begin{figure}[H]
\begin{subfigure}[b]{0.34\textwidth}
\includegraphics[width=\textwidth]{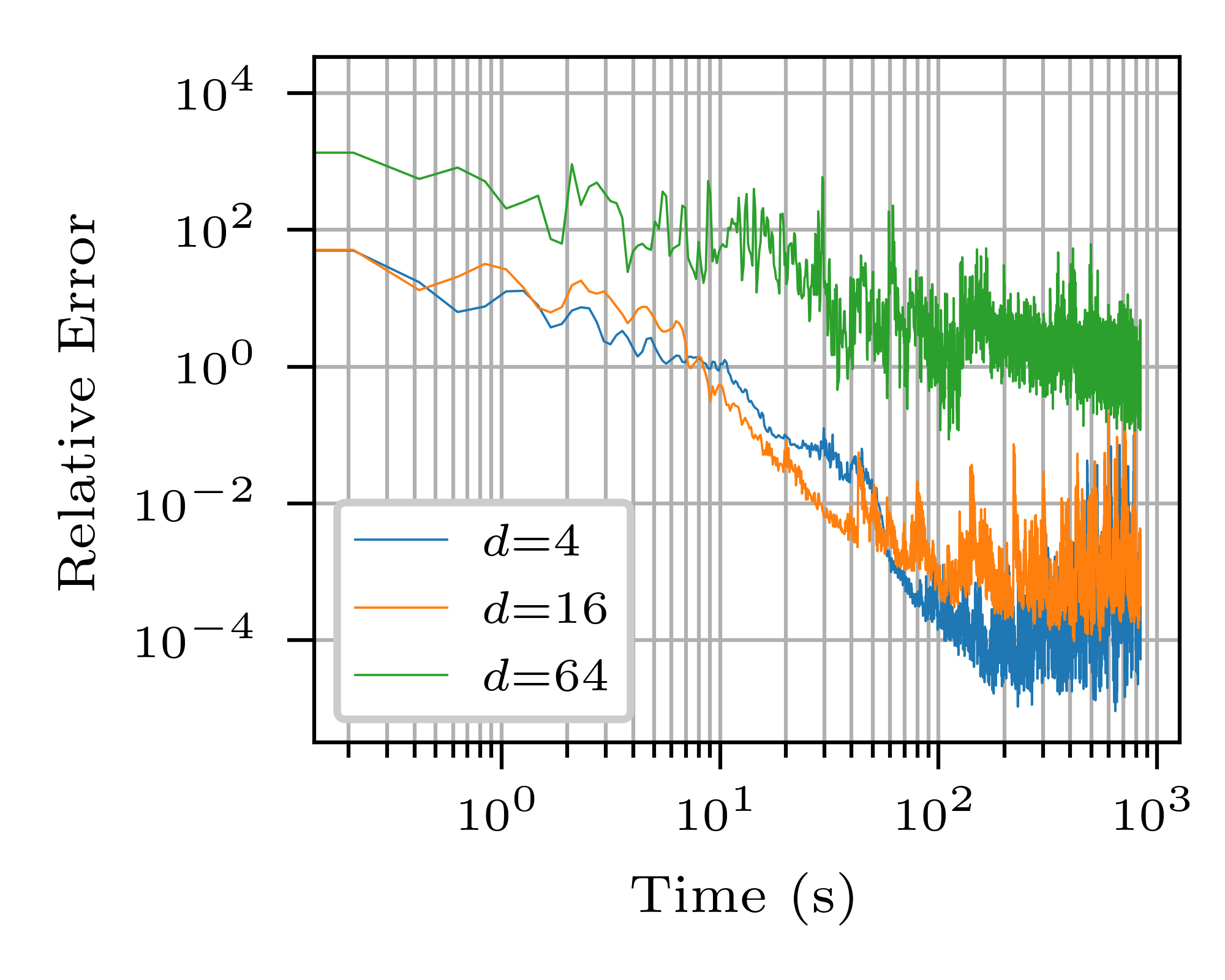}
\subcaption{The relative error over time}\label{fig:scalability1}
\end{subfigure}\hspace{0.0cm}
\begin{subfigure}[b]{0.65\textwidth}
\includegraphics[width=\textwidth]{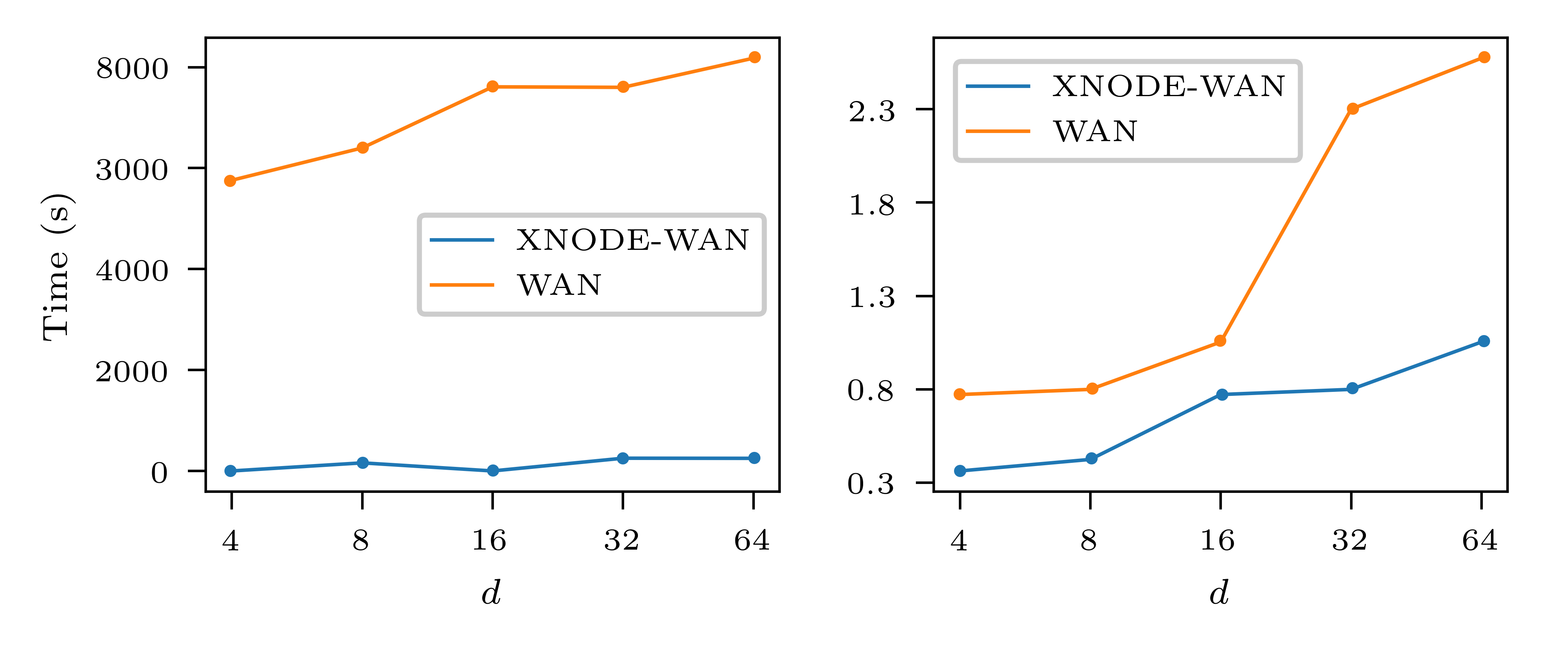}
\subcaption{Total computation time and training time per epoch}\label{fig:scalability2}
\end{subfigure}
    \caption{Scalability results of \Cref{Ex_sensitivity} in terms of spatial dimension $d$. (a) the evolution of relative $L^2$ error of the XNODE-WAN method over time ($0$ to $10^3$s) for $d \in \{4, 16, 64\}$; (b) the computational time costs (Left) and average training time per epoch (Right) of XNODE-WAN and WAN until they achieve error tolerance respectively. Here we use $\epsilon=10^{-2}$ for each spatial dimension $d \in \{4, 8, 16, 32, 64\}$.
    }
    \label{fig:scalability}
\end{figure}

\subsection{Sensitivity analysis}\label{sec:sensitivity}
In this subsection, we aim to investigate the sensitivity of our proposed XNODE model in terms of the network architecture of the hidden field $\mathcal{N}^{\text{vec}}_{\theta_2}$ and learning rate $\tau_\theta$. Here we choose the same PDE setting defined in Section \ref{sec:scalability} and set $d = 5$. We follow \cite{zang2020weak} to choose the error tolerance level $\epsilon = 0.005$.

To investigate the effects of the neural network architecture of the hidden field $\mathcal{N}^{\text{vec}}_{\theta_2}$ on the predictive performance of the XNODE-WAN, we consider different number of neurons per layer (hidden dimension) and number of layers (depth). When the hidden dimension is high as $18$ or $22$, our method can achieve the satisfactory results. It can be observed, when increasing the depth to $7$ or above, the predictive performance of our method is improving. It suggests that our method usually works well as long as the neural network architecture is sufficiently rich. However, the over-complicated network should be discouraged as it may prolong the training time and cause the overfitting.

\begin{figure}[H]
\begin{subfigure}[b]{0.36\textwidth}
\includegraphics[width=\textwidth]{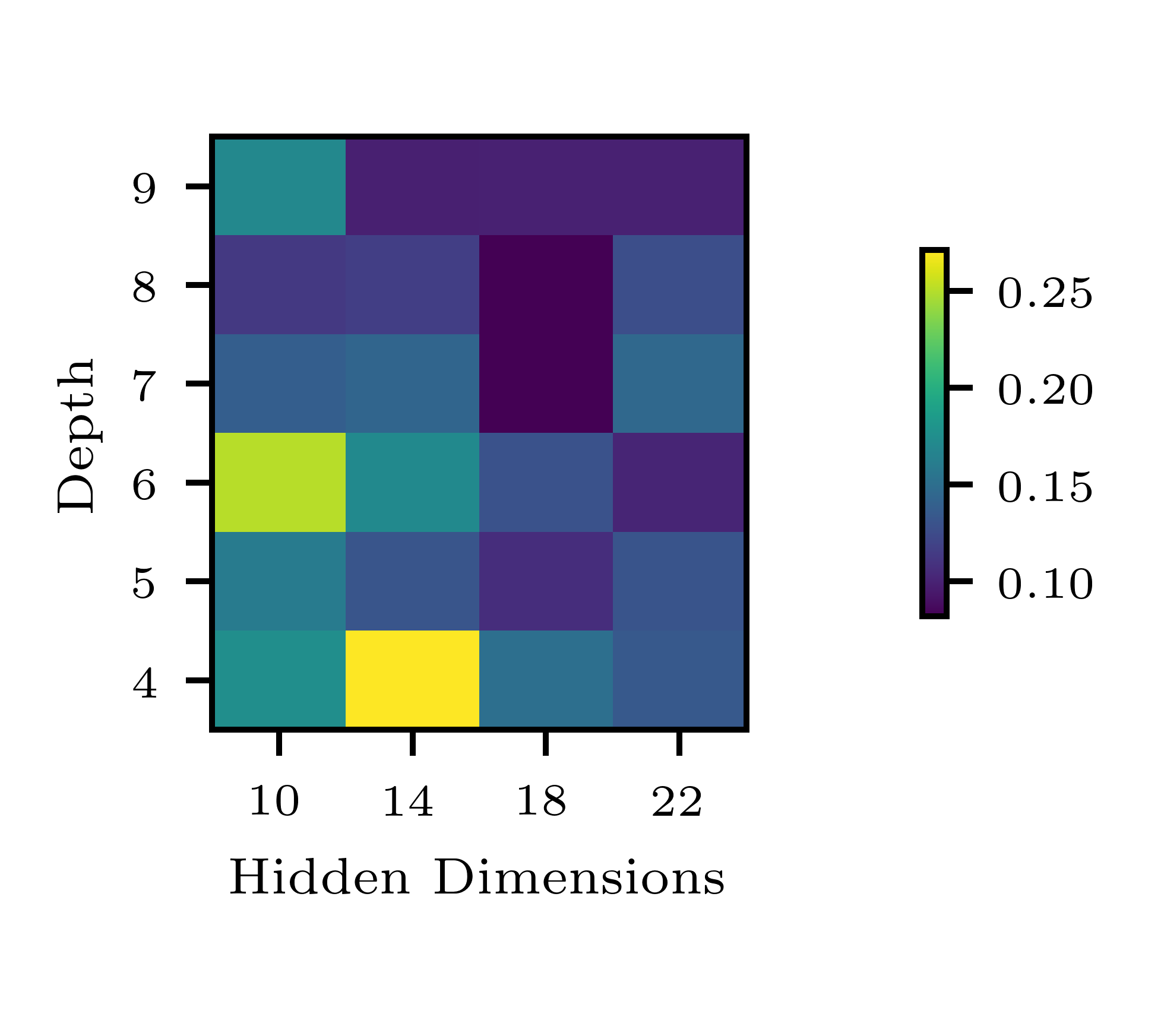}
\subcaption{XNODE-WAN: depths against hidden dimensions}\label{fig:sensitivity1}
\end{subfigure}\hspace{0.1cm}
\begin{subfigure}[b]{0.63\textwidth}
\includegraphics[width=\textwidth]{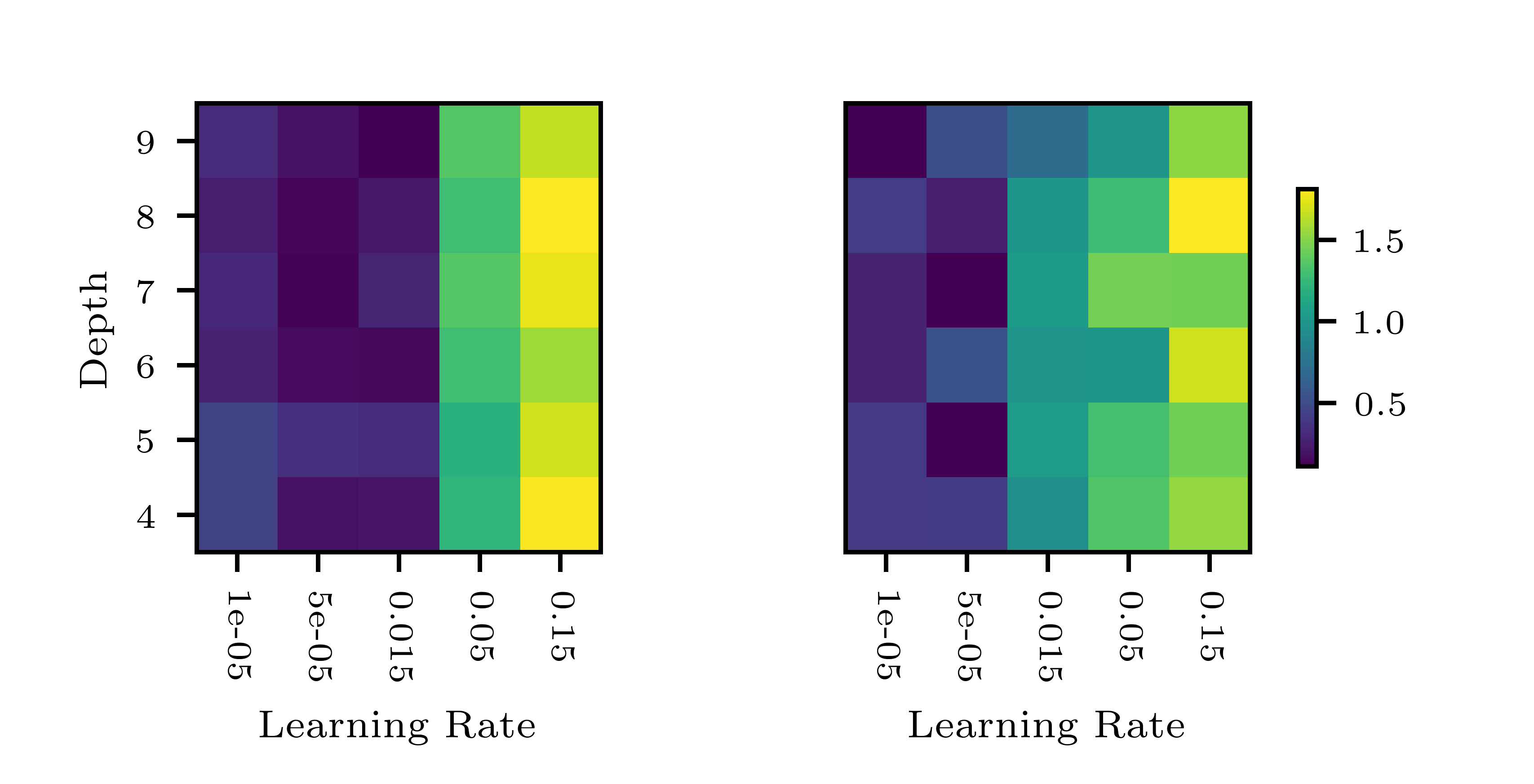}
\subcaption{XNODE-WAN (right) and WAN (left): the heatmap of relative error against the network depth and learning rate}\label{fig:sensitivity2}
\end{subfigure}\hspace{0.2cm}
    \caption{The relative error plots for sensitivity analysis. 
    All these experiments are implemented under the same setting of hyperparameters as in the Section \ref{sec:scalability}.}
    \label{fig:sensitivity}
\end{figure}

Furthermore, we consider different values of learning rate for both our method and WAN, and the corresponding performance. Fig. \ref{fig:sensitivity2} shows that, independently of the depth of the hidden field, a learning rate on the magnitude of $10^{-5}$ is needed to achieve a reasonable loss for WAN while XNODE-WAN allows for a much wider range of learning rate, i.e., from $10^{-5}$ to $10^{-2}$. In practice, we conduct the grid search to select the optimal learning rate for each method, i.e. to achieve the satisfactory accuracy, the learning rate is chosen to be the one corresponding to the least training time. 
That is the justification of different learning rate used for XNODE-WAN and WAN.

\subsection{Time-varying domain}\label{sec:num_timevarying}
\begin{example}\label{ex3}
We now consider the PDE problem on the time-varying domain $\mathcal{D} \subset [0,1] \times[-1,1]$ defined as follows (see \Cref{fig:time-varying1} for its shape),
\begin{eqnarray*}
\mathcal{D} &= \left\{(t,x_1)| t \in [0, 0.5], |x_1 -0.5|\leq 0.5 (1-t) \right\} \\
&\bigcup \left\{(t,x_1) | t \in [0.5, 1], |x_1 -0.5|\leq 0.5 t \right\}.
\end{eqnarray*}
 Note that to illustrate the idea in this example we only consider $d=1$. 
 \end{example}

\begin{figure}[h]
\centering
\begin{subfigure}[b]{0.35\textwidth}
\includegraphics[width=\textwidth]{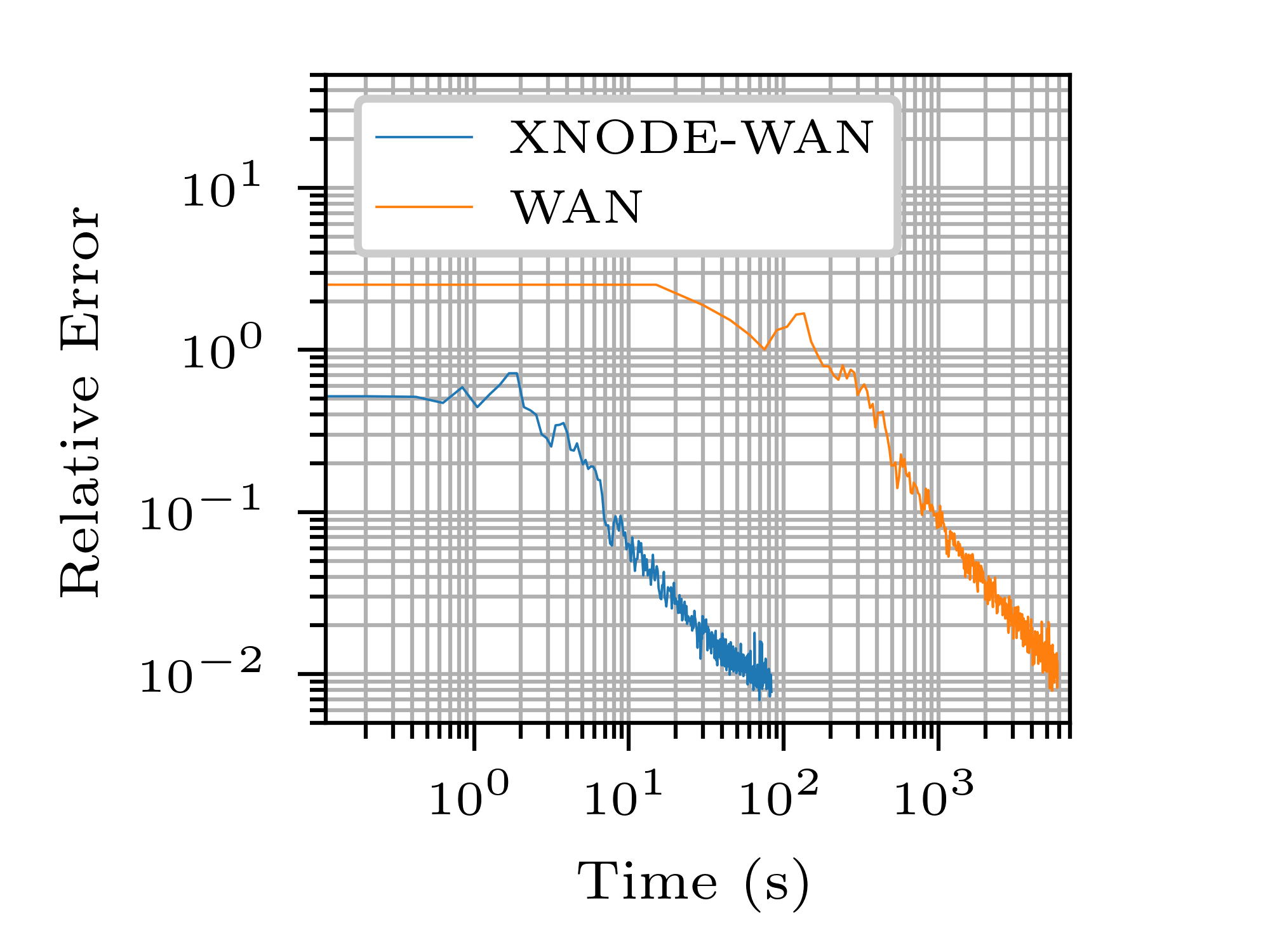}
\subcaption{The relative error over time }\label{fig:time-varying1}
\end{subfigure}
\begin{subfigure}[b]{0.64\textwidth}
\includegraphics[width=\textwidth]{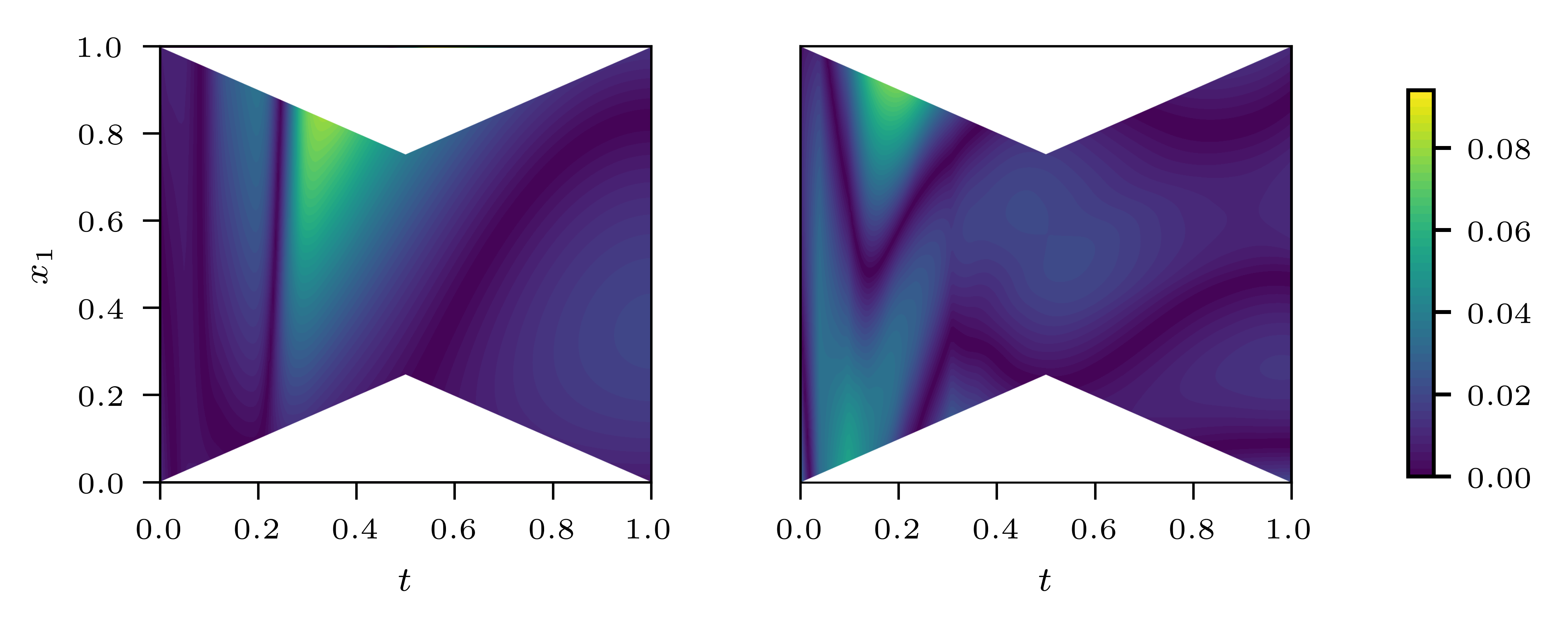}
\subcaption{Error at convergence for XNODE-WAN and WAN solutions}\label{fig:time-varying2}
\end{subfigure}
\caption{Performance comparison of XNODE-WAN and WAN on an time-dependent domain in \Cref{ex3}. (a) the evolution of relative $L^2$ error of both XNODE-WAN and WAN method over time; (b) The plot of poitwise absolute error $|u(t, x_1) - \hat{u}(t, x_1)|$, where $\hat{u}$ is the trained model of each method until hitting the stopping criteria.}
\label{fig:time-varying}
\end{figure}

In this example, we choose the error tolerance level $\epsilon=10^{-2}$ to train both XNODE-WAN and WAN models. The performance comparison results are shown in \Cref{fig:time-varying}. In particular, from \Cref{fig:time-varying2} we observe that both models are able to learn the solution on the time-varying domain to a satisfactory level. The heatmap plot of Fig. \ref{fig:time-varying2} shows that they both achieve an error below 0.04 for most of the region at their convergence. However, as shown in Fig.\ref{fig:time-varying1}, similarly to the time-independent case, the convergence of XNODE-WAN in terms of the relative error is much faster than that of WAN, which is reflected by the enlarged relative error gap between two models over time until the end of XNODE-WAN training. 





To solve the PDE problem on the time-varying domain, both WAN and XNODE-WAN may require more computational time to train compared with their time-independent counterpart as expected. It is evident by that the computational cost of each model on low-dimensional time varying domain example (\Cref{ex3}) is even higher than that for the higher dimensional time-invariant domain (see Fig. \ref{fig: sol_training_loss1}). For our method, it is mainly a result of additional complexity of constant sub-path construction in Algorithm \ref{NODE-timevary}.

However, more importantly, for this time-varying domain example, our proposed model consistently outperforms in terms of efficiency: WAN requires significantly longer time (5,000s) to converge compared to XNODE-WAN (less than 100s) in Fig. \ref{fig:time-varying1}. In part this can be explained by the fact that XNODE-WAN model incorporates the boundary condition to the model (Eq. \eqref{Time_varying_eqn}) and its weak solution estimators evaluated at the entry points on the boundary have little error. This anchors our model at more places than the WAN model, making it more suited to a time-varying domain. This example justifies the advantage of using XNODE-WAN for solving PDE on an time-varying domain.


\section{Conclusion}\label{sec: Conclusion}
In this paper, we propose the novel XNODE model as a universal and effective network for the parabolic PDE solution by leveraging the priori information of the PDE though the NODEs. To tackle the high dimensional parabolic PDE problems on arbitrary bounded domain, incorporating the XNODE model to the WAN as a primal network leads to significantly faster training and better scalability. Our XNODE-WAN method consistently outperforms WAN in terms of computational time on both time-independent domain and time-dependent domain. 

\section*{Acknowledgements}
Funding: HN and YW are supported by the Engineering and Physical Sciences Research Council (EPSRC) [grant number EP/S026347/1] and the Alan Turing Institute under the EPSRC grant [grant number EP/N510129/1]. \newpage

\appendix

\section{Universality of XNODE model}
\label{appendix:univerality}

Let us recall the definition of XNODE model, which is a map from an input $\mathbf{x} \in \mathbb{R}^d$ to the output $o = (o_{t})_{t \in [0, T]} \in \mathcal{C}([0, T], \mathbb{R})$ through the $\mathbb{R}^{h}$-valued latent process $\mathbf{h} = (\mathbf{h}(t))_{t \in [0, T]}$, i.e. for every $ t \in [0, T]$,
 \begin{eqnarray}
 \begin{cases}
 \frac{\mathrm{d}\mathbf{h}(t)}{\mathrm{d}t} = \mathcal{N}^{\text{vec}}_{\theta_1}(\mathbf{h}(t), t, x), \quad \mathbf{h}(0) = \mathcal{N}^{\text{init}}_{\theta_{2}}(h(x)) \in \mathbb{R}^h.& \\
o_{t} = \mathcal{L}_{\tilde{\theta}}(\mathbf{h}(t)). &
 \end{cases} \label{xODE-eqn}
 \end{eqnarray}
where $\mathcal{N}^{\text{vec}}_{\theta_1}$ and $\mathcal{N}^{\text{init}}_{\theta_{2}}$ are neural networks fully parameterized by $\theta_{1}$ and $\theta_2$ for the vector fields and initial condition of the hidden neural $\mathbf{h}$ respectively, and $\mathcal{L}_{\tilde{\theta}}$ is a linear trainable layer with the weights $\tilde{\theta}$.

We will establish the universal approximation theorem of the proposed XNODE model for the parabolic PDE solution. Our proof mainly replies on the stability estimates of the ODE theory.

\begin{theorem}[Theorem 2.8 in \cite{teschl2012ordinary}]\label{ODE_stability_thm}
Suppose $V_{1}, V_2 \in \mathcal{C}(U,\mathbb{R}^n)$ where $U$ is an compact subset in $[0,T] \times \mathbb{R}^{n}$ and $V_{1}$ is Lipschitz  continuous in the first argument with 
\begin{equation*}
    L = \sup_{\overset{(\mathbf{y}_1, t),(\mathbf{y}_2,t) \in U}{\mathbf{y}_1 \neq \mathbf{y}_2 }}\frac{|V_1( \mathbf{y}_1,t) - V_1(\mathbf{y}_2,t)|}
{|\mathbf{y}_1 - \mathbf{y}_2|}.
\end{equation*}
If $F_1$ and $F_{2}$ are the solutions to the two initial value problems defined as follows:
\begin{eqnarray}
&&F_1'(t) = V_1\big(F_1(t), t\big),\quad F_1(0) = f_1;\\
&&F_2'(t) = V_2\big(F_2(t), t\big),\quad F_2(0) = f_2.
\end{eqnarray}
Then for all $t \in [0, T]$,
\begin{eqnarray*}
|F_1(t) - F_2(t)| \leq |f_1 - f_2|\exp(Lt) + \frac{M}{L}\big(\exp(Lt) -1\big),
\end{eqnarray*}
where
\begin{eqnarray*}
M = \sup_{(\mathbf{y},t) \in U} |V_1( \mathbf{y}, t) - V_2(\mathbf{y}, t)|.
\end{eqnarray*}
\end{theorem}

Thanks to the universality of neural network and the above theorem, we now establish the universality of the XNODE model in the following theorem.

\begin{theorem}\label{thm_universality}
Let $\Omega \subset \mathbb{R}^d$ be a bounded domain and $F: [0, T] \times \Omega \rightarrow \mathbb{R}$
be a solution to the following ODE with the initial condition:
\begin{eqnarray}\label{ODE}
\frac{\mathrm{d}F(t,\mathbf{x})}{\mathrm{d}t} = V(F(t,\mathbf{x}), t; \mathbf{x}), \quad F(0,\mathbf{x}) = h(\mathbf{x}),
\end{eqnarray}
where $V$ is Lipschitz continuous in the first argument for any $\mathbf{x} \in \Omega$ and there exists a bounded constant $L>0$ such that 
\[
    L =  \sup_{\mathbf{x} \in \Omega}
  \sup_{\overset{(\mathbf{y}_1, t), (\mathbf{y}_2, t) \in U}{ \mathbf{y}_1 \neq \mathbf{y}_2}}\frac{|V( \mathbf{y}_1, t; \mathbf{x}) - V(\mathbf{y}_2, t;\mathbf{x})|}
{|\mathbf{y}_1 - \mathbf{y}_2|}.
\]
Then for any given error tolerance level $\epsilon >0$, there exists a XNODE solution $\tilde F$ defined in \eqref{xODE-eqn} such that 
\begin{eqnarray}\label{stability}
\sup_{(t, \mathbf{x}) \in  [0, T] \times \Omega}|F(t,\mathbf{x}) - \tilde F(t,\mathbf{x})| \leq \epsilon .
\end{eqnarray}
\end{theorem}
\begin{proof}
In this proof, we constrain the XNODE models to the ones which has the hidden dimension 1 and identity map being their output layer. Therefore $\tilde{F}$ satisfies the below ODE similar to \eqref{ODE}:
\begin{eqnarray}\label{ODE-2}
\frac{\mathrm{d}\tilde F(t,x)}{\mathrm{d}t} = \mathcal{N}^{\text{vec}}_{\theta_1}(\tilde F, t; x), \quad \tilde F(0,x) = \mathcal{N}^{\text{init}}_{\theta_2}(h(x)),
\end{eqnarray}
where the activation functions (e.g. ReLu) are chosen such that the vector field is uniformly Lipschitz continuous. the wellposeness of the XNODE model is a directly consequence of Picard's Theorem. So is the function $F$ due to the Lipschitz continuity assumption of its vector field $F$.   

Thanks to the universality of neural network, for every $\epsilon > 0$, there exist two neural networks with sufficiently large number of hidden neurons and number of layers, denoted by $\mathcal{N}^{\text{vec}}_{\theta_1}$ and $\mathcal{N}^{\text{init}}_{\theta_2}$  such that 
\[
\sup_{\mathbf{y},t,\mathbf{x}}  |\mathcal{N}^{\text{vec}}_{\theta_1}(\mathbf{y},t;\mathbf{x}) - V(\mathbf{y},t;\mathbf{x})| + 
\sup_{\mathbf{x}} |\mathcal{N}^{\text{init}}_{\theta_2}(h(\mathbf{x}))- h(\mathbf{x})| 
\le C^{-1} \epsilon,
\]
where $C$ can be pre-determined through
\[
C = \exp{(LT)} + \exp{(LT)}/L.
\]
The stability result \eqref{stability} is then a direct consequence of Theorem A.\ref{ODE_stability_thm}.

\end{proof}

\section{Pseudocode of XNODE-WAN Algorithm}
\begin{table}[H]
\centering
\resizebox{\columnwidth}{!}{%
\begin{tabular}{l|l}
\hline
         Notation  & Meaning \\
         \hline
         $X_{\mathcal{D}}$ & Sampled collocation points of the whole domain $\mathcal{D}$ \\
$X_{\partial\mathcal{D}}$ & Sampled collocation points on the domain boundary $\partial\mathcal{D}$ \\
$X_0$ & Sampled collocation points at the initial condition  \\
$S^r$ & Sampled collocation points of the spatial domain  \\
$S^b$ & Collocation points of the spatial domain boundary\\
$\Pi_T$ & Sampled time partition \\
$N_\mathcal{D}$      & Number of sampled collocation points of the space-time domain $\mathcal{D}$\\
$N_{\partial \mathcal{D}}$      & Number of sampled collocation points of domain boundary $\partial\mathcal{D}$\\
$N_0$     & Number of sampled collocation points of the initial condition \\
$N_r$      & Number of sampled collocation points of the spatial domain \\
$N_b$ &  Number of sampled collocation points of the spatial domain boundary  \\
$n_T$ &  Number of sampled time partition \\
$K_u$ &  Inner iteration to update weak solution $u_\theta$ or $u_\Theta$ \\
$K_\phi$ &  Inner iteration to update test function $\phi_\eta$  \\
$\alpha$ &   Weight parameter of boundary loss $\tilde{L}_\text{bdry}$\\
$\gamma$ &  Weight parameter of initial loss $\tilde{L}_\text{init}$\\
$\epsilon$ & error tolerance\\
$\tau_\theta$ or $\tau_\Theta$ & Learning rate for the primal network \\
$\tau_\eta$ & Learning rate for network parameter $\eta$ of test function $\phi_\eta$\\

\hline
\end{tabular}}
\caption{List of algorithm parameters.}
\label{table2}
\end{table}

We present in the Table \ref{table2} the explanation for notations in algorithms. The following algorithm (Algorithm \ref{XNODE-WAN_indep_domain}) is the full pseudocode of the XNODE-WAN algorithm for a time-independent domain.
\begin{algorithm}[H]
\caption{XNODE-WAN Algorithm for time-independent domains}\label{XNODE-WAN_indep_domain}
\textbf{Input:} domain $\mathcal{D} =
\Omega \times [0,T]$, tolerance $\epsilon>0$, $N_r/N_b/n_T$: number of sampled points on spatial domain/spatial boundary/temporal domain, $K_u/K_\phi$: number of solution/adversarial network parameter update per iteration, $\alpha/\gamma \in \mathbb{R}^+$: the weight of boundary/initial loss 
\begin{algorithmic}[1]
\STATE{\textbf{Initialise:} the XNODE network $\phi_\Theta:\mathcal{D} \to \mathbb{R}$ and the DNN network $\phi_\eta:\mathcal{D} \to \mathbb{R}$
 }
\STATE{generate the sets of collection points $S^r, S^b$ and $\Pi_T$  } \\
\algorithmiccomment{random point sampling step}

\WHILE{$\tilde L(\Theta,\eta)>\epsilon$} 
\STATE{\# \texttt{update weak solution network parameter}}
\FOR{$k = 1,...,K_u$}
\STATE{compute $\nabla_\theta \tilde L(\Theta, \eta)$;}\\ \algorithmiccomment{gradient calculation step }
\STATE{update $\Theta\gets \Theta- \tau_\Theta \nabla_\Theta \tilde L(\Theta, \eta)$;} 
\ENDFOR
\STATE{\# \texttt{update test network parameter}}
\FOR{$k = 1,...,K_\phi$}
\STATE{compute $\nabla_\eta \tilde L(\Theta, \eta)$ using the points $X^r$;}
\STATE{update $\eta\gets\eta+\tau_\eta \nabla_\eta \tilde L(\Theta, \eta)$;} 
\ENDFOR
\STATE{generate the sets of collocation points  $S^r$, $S^b$ and $\Pi_T$;} \\\algorithmiccomment{random point sampling step}

\ENDWHILE\\
\textbf{Output:} $\{\mbox{XNODE}_\Theta(\mathbf{x}_i,\Pi_T)\}_{\mathbf{x}_i \in S^r \cup S^b}$
\\\algorithmiccomment{the weak solution estimator evaluated at $\{x\} \times \Pi_T$}
\end{algorithmic}

\end{algorithm}

\section{Numerical Results Summary}

Below we present some results from the comparisons of the XNODE-WAN and WAN models presented in \Cref{ex2} and \Cref{ex3} .

\begin{table}[H]
\centering
\resizebox{\columnwidth}{!}{
\begin{tabular}{|l|l|l|l|l|}
\hline
\multicolumn{5}{|c|}{ \Cref{ex2}: Cylinder Domain  }\\
\hline
  Model & Relative error  & Time per epoch (s) & $N_\varepsilon$ &  $T_\varepsilon$ (s)\\
\hline
  WAN &  1.0\%  &  0.38 & 15,278 & 5806 \\ 
 XNODE-WAN & 1.1\%  & 0.35 &  271  &  88 \\ 
\hline
\multicolumn{5}{|c|}{\Cref{ex3}: Hourglass domain }\\
\hline
  WAN & 8.3\%  &  0.38 & 21,594 & 8206 \\ 
 XNODE-WAN & 7.1\% & 0.43 &  221  &  95 \\
\hline
\end{tabular}}
\caption{Performance comparison between XNODE-WAN and WAN on various domains for \Cref{ex2} and \Cref{ex3}}\label{Tab1_eg1}.
\label{table}
\end{table}

\pagebreak
\bibliographystyle{elsarticle-num-names}
\bibliography{bibfile}

\end{document}